\documentclass[11pt, english, a4paper]{amsart}

\usepackage{graphicx}

\graphicspath{{img/}}
\usepackage{amsmath}
\usepackage{amssymb}
\usepackage[utf8]{inputenc}
\usepackage[colorlinks=true, citecolor=black, linkcolor=black, urlcolor=blue]{hyperref}
\usepackage{bm} 
\usepackage[T1]{fontenc}
\usepackage{amsfonts}
\usepackage{mathrsfs}
\usepackage{yfonts}
\usepackage{url}
\usepackage[normalem]{ulem}
\usepackage{mathpazo}  

\usepackage{mathtools}

\usepackage{verbatim}
\usepackage{enumerate}	
\usepackage[english]{babel}
\usepackage{ebproof}		
\usepackage{quoting}
\usepackage{array,booktabs}
 \usepackage{pifont}

\usepackage[shortlabels]{enumitem}

\usepackage{todonotes}

\usepackage[pdf]{graphviz}

\usepackage[margin=2.5cm]{geometry}

\usepackage{orcidlink}

\theoremstyle{definition}
\newtheorem{definition}{Definition}
\newtheorem{example}[definition]{Example}

\newtheorem{remark}[definition]{Remark}

\theoremstyle{plain}
\newtheorem{lemma}[definition]{Lemma}
\newtheorem{proposition}[definition]{Proposition}
\newtheorem{theorem}[definition]{Theorem}
\newtheorem{corollary}[definition]{Corollary}

\makeatletter
\newcommand{\tpmod}[1]{{\@displayfalse\pmod{#1}}}
\makeatother

\newcommand\Ni[1][S]{\mathfrak{I}_{0}(#1)}
\newcommand{\Ap}{\operatorname{Ap}}

\title{Reflexive ideals of a numerical semigroup}
\author[Campochiaro]{P. Campochiaro \orcidlink{0009-0009-5787-1898}}
\address[Campochiaro]{Institute of Mathematics, Jagiellonian University, prof. Stanisława Łojasiewicza 6, 30-348 Kraków, Poland}
\email{pietro.campochiaro@doctoral.uj.edu.pl}

\author[D'Anna]{M. D'Anna \orcidlink{0000-0002-9943-5527}}
\address[D'Anna]{Dipartimento di Matematica e Informatica, Viale A. Doria, 6 - I 95125 – Catania, Italia}
\email{mdanna@dmi.unict.it}

\author[García-Sánchez]{P.~A. García-Sánchez \orcidlink{0000-0003-2330-9871}}
\address[García-Sánchez]{Departament of Algebra and IMAG, Univsersity of Granada, 18071 Granada, Spain}
\email{pedro@ugr.es}


\date{January 2025}

\begin{document}

\begin{abstract}
    We study the reflexivity property of ideals in numerical semigroups. In particular, we characterize irreducible reflexive ideals and  determine the Ap\'ery set of a reflexive ideal. Moreover, we study how reflexivity allows us to characterize some classes of numerical semigroups and the interaction of reflexivity with the property of being a trace ideal.
\end{abstract}

\maketitle

\section{Introduction}
 In commutative ring theory, a module $M$ over a ring $R$ is said to be \emph{reflexive} when the function $\varphi: M \rightarrow \operatorname{Hom}_R(\operatorname{Hom}_R(M,R),R)$, defined by $\varphi(x)(f)=f(x)$, is an isomorphism.
If $I$ is a regular fractional ideal of $R$, we have
$\operatorname{Hom}_R(I,R)\cong (R: I)$, where the colon is taken in the total ring of fractions $Q(R)$ of $R$, and therefore reflexivity can be expressed as $I=R:(R:I)$.
In this case, $I$ is also said to be \emph{divisorial}.

The interactions between reflexivity of modules and properties of the ring have been an object of research since the 1960's; recently, many papers  dealt with reflexivity for one-dimensional rings associated to curve singularities (see, for instance, \cite{D,DMS,IK24,Kob}).

A particular case of this class of rings is obtained by considering numerical semigroup rings (or, equivalently, affine coordinate rings of a monomial curve).
In this set-up, it is possible to translate at numerical level the multiplicative properties  of the rings, looking at additive properties of the corresponding semigroup.
This technique not only allows to recover more easily known results that hold for rings, but often yields original results, which in many cases can be generalized for lager classes of rings. Moreover, working at the numerical level is useful to find examples and counterexamples for any kind of pathology.

Hence, in this paper, we focus on reflexivity problems on numerical semigroups, that is, submonoids $S$ of $(\mathbb N,+)$ (where $\mathbb N$ denotes the set of non-negative integers), with finite complement in $\mathbb{N}$, and on  their relative ideals $I$,
that is, non-empty subsets of $\mathbb Z$, with a minimum and closed under addition with elements of $S$. In this context, the reflexivity of an ideal $I$ can be expressed as 
\[
I=S-(S-I),
\]
where the ``difference'' of two relative ideals is the numerical counterpart of the colon ideal
in ring theory: $H-J=\{x \in \mathbb Z\mid x+j \in H, \ \text{for all } j \in J\}$.

The main problems we address are of two different natures.
 Firstly we try to characterize reflexive ideals of a numerical semigroup, a problem that for rings is still open. To this aim we focus our attention on factorizations of a reflexive ideal
 as intersection of other reflexive ideals, determining the irreducible ones. We consider two different situations: the irreducible in the set of all reflexive ideals, that turns out to be the principal ideals, and the irreducible in the set of ideals between $S$ and $\mathbb N$, which turns out to be translations of integrally closed ideals.
 Secondly, we focus on the  interactions between reflexive ideals and the properties of the semigroup; in this case, we recover many results  obtained in the ring case, but with original proofs and new interesting examples that enlighten the theory.

More precisely, after recalling the basic definitions and properties needed for our goals, 
in Section~\ref{sed:reflexive-ideals} we start to study the divisorial closure of a relative ideal, that is, $I^v:=S-(S-I)$.
Then we prove that every reflexive ideal is a finite intersection of principal ideals (see Proposition \ref{prop:expr-I-star}), that are the only irreducible reflexive ideals (Proposition \ref{prop:p ideals}), and we discuss the possible lengths of these factorizations.
Moreover, after observing that reflexivity is independent of translations, we show that any reflexive ideal is a translation of a proper ideal sharing the conductor with $S$. We also characterize when an over-semigroup of $S$ and all its ideals are reflexive (see Theorem~\ref{thm:A-Valentina}).

If we focus on ideals sharing the conductor with $S$, we can find a necessary condition to reflexivity (see Corollary~\ref{cor:m+bar I-bar I}), which involves the notion of integrally closed ideals,  a class of proper ideals, that are all reflexive, and this is the content of Section~\ref{sec:int-closed-ideals}.

However, in order to find a characterization of reflexive ideals, we also consider the monoid of normalized ideals, which is the set of relative ideals  $I$, such that $\min(I)=0$.
In Section~\ref{sec:irred-ref-ideals} we show that every normalized reflexive ideal can be expressed as a finite intersection of normalized irreducible reflexive ideals.
Moreover, we study the nature of these decompositions, showing that they are not unique and can be of different length. In particular, the normalized irreducible reflexive ideals are exactly the translations of the integrally closed ideals of $S$ (see Theorem~\ref{th:R-irr}). 

In Section~\ref{sec:apery-sets}, we deepen our object of study using Ap\'ery sets. Firstly, we show that they allow us to determine the irreducible reflexive ideals appearing in a minimal decomposition of a reflexive ideal; then, after computing the Ap\'ery set of the difference ideal of two ideals (see Theorem~\ref{th:ap-smi}), we use this fact to determine the Ap\'ery set of the reflexive closure of an ideal (see Corollary~\ref{cor:ap-smsmi}). 

In the last two sections, we deal with the second 
goal of our paper, that is, the interactions between reflexivity and properties of the semigroup. 
More precisely, in Section
\ref{sec:some-families}
we characterize reflexive ideals for almost symmetric, maximal embedding dimension and Arf numerical semigroups, recovering results yet proved in \cite{CDS} for 
one dimensional rings. 
For these three classes, we
show that they can be characterized by the
nature of their reflexive ideals. 
Finally, in Section~\ref{sec:trace-ideals},
we deal with trace ideals, 
a class of ideals connected with that of reflexive ideals, but not included nor including it. In particular, after proving a one to one correspondence between reflexive trace ideals and reflexive oversemigroups (that is a numerical version of a result in \cite{GIK20}), we give two new characterization of Arf rings in terms of trace ideals.


\section{Preliminaries on numerical semigroups and ideals}\label{sed:prelim}

In this section, we fix the notation and give the basic notions on numerical semigroups and their ideals, which we will need throughout the rest of the paper.

Let $\mathbb{N}$ denote the set of non-negative integers. A \emph{numerical semigroup} is a submonoid of $(\mathbb{N},+)$ with finite complement in $\mathbb{N}$. The monoid $\mathbb N$ is itself a numerical semigroup, but to avoid trivialities, we will always assume that $S\neq \mathbb N$.

A non-empty set of integers $I$ is an \emph{ideal} (or relative ideal) of $S$ if $I+S=I$ and there is an integer $z$ such that $z+I\subseteq S$ (this last condition is equivalent to the existence of $\min(I)$). An ideal $I$ of a numerical semigroup $S$ is \emph{integral} if $I\subseteq S$. In particular, $M=S\setminus\{0\}$ is an integral ideal, known as the \emph{maximal ideal} of $S$. The set of ideals of $S$ will be denoted by $\mathfrak{I}(S)$.
For an introduction to ideals of numerical semigroups, the reader may refer to \cite[Chapter~3]{ns-app}.

If $S$ is a numerical semigroup, the fact that $\mathbb{N}\setminus S$ has finitely many elements implies that there is an integer $c$ such that $c+\mathbb{N}\subseteq S$. The minimum integer with this property, known as the \emph{conductor} of $S$, is denoted by $\operatorname{c}(S)$. Clearly, $\operatorname{c}(S)+\mathbb{N}$ is an integral ideal of $S$, known as the \emph{conductor ideal} of $S$, denoted by $\operatorname{C}(S)$.
Given an ideal $I$ of a numerical semigroup $S$,  from the fact that $z+I\subseteq S$ for some $z\in \mathbb{Z}$, we derive that there is a minimum integer $\operatorname{c}(I)$ such that $\operatorname{c}(I)+\mathbb{N}\subseteq I$, which by analogy is called the \emph{conductor} of $I$. We define $\operatorname{C}(I)=\operatorname{c}(I)+\mathbb{N}$.

If $I$ and $J$ are ideals of $S$, we say that $I$ is a \emph{translate} of $J$ if there exists an integer $z$ such that $I=z+J$ (notice that in this case $J$ is also a translate of $I$).

Every ideal $I$ of $S$ has a unique translate $\widetilde{I}$ such that $\operatorname{c}(\widetilde{I})=\operatorname{c}(S)$, namely 
\[
\widetilde{I}=(\operatorname{c}(S)-\operatorname{c}(I))+I.\]

Let $\operatorname{f}(S)=\max(\mathbb{Z}\setminus S)$ (its existence is guaranteed by the fact that $S$ has finite complement in $\mathbb{N}$), which is known as the \emph{Frobenius number} of $S$. Notice that $\operatorname{c}(S)=\operatorname{f}(S)+1$. For an ideal $I$ of $S$, we define, in a similar way, its Frobenius number as $\operatorname{f}(I)=\max(\mathbb{Z}\setminus I)$.

The Frobenius number of a numerical semigroup plays a key role in the theory of numerical semigroups. For example, it is used in the definition of the  \emph{canonical ideal} of $S$:
\[
\operatorname{K}(S)=\{ z\in \mathbb{Z} : \operatorname{f}(S)-z\not\in S\}.
\]
By definition, $\operatorname{K}(S)$ is an ideal of $S$ with $S\subseteq \operatorname{K}(S)\subseteq \mathbb{N}$. 
We say that $S$ is \emph{symmetric} if $\operatorname{K}(S)=S$.

To every numerical semigroup $S$ we can associate the following ordering on $\mathbb{Z}$: $a\le_S b$ if $b\in a+S$.

Given an ideal $I$, we say that it is generated by a set $A\subset I$, if all the elements of $I$ are of the form $a+s$, with $a\in A$ and $s \in S$. In this case, we write $I=A+S$. If we consider the minimal elements of $I$ with respect to $\le_S$, we have
\[
I=\operatorname{Minimals}_{\le_S}(I)+S.
\]
The set $\operatorname{Minimals}_{\le_S}(I)$ is the \emph{set of minimal generators} (or minimal generating set) of $I$. Since every element in this set is in a different congruence class modulo $\min(S\setminus\{0\})$, we deduce that $|\operatorname{Minimals}_{\le_S}(I)|\le \min(S\setminus\{0\})$. 
An ideal $I$ of $S$ is \emph{principal} if $I=z+S$ for some integer $z$. 
The integer $\min(S\setminus\{0\})$ is known as the \emph{multiplicity} of $S$, denoted by $\operatorname{m}(S)$.

The sum, intersection, and union of ideals of a numerical semigroup are again ideals of the same numerical semigroup. Moreover, if $I,J$ are ideals of a numerical semigroup $S$, the set 
\[I-J = \{x\in\mathbb{Z} : x+ J\subseteq I\}\]
is again an ideal of $S$. 

Let $S$ be a numerical semigroup. It is not difficult to prove that for $I,J,K\in \mathfrak{I}(S)$ and $x\in \mathbb{Z}$ the following properties, concerning the operations mentioned above, hold:
\begin{enumerate}[({D }1)]
    \item \label{D-1}$(x+I)-J=x+(I-J)$ and $I-(x+J)=-x+(I-J)$;
    \item \label{D-2}if $J\subseteq K$, then $I-J \supseteq I-K$ and $J-I\subseteq K-I$;
    \item \label{D-3} $I-(J\cap K)\supseteq (I-J)\cup (I-K)$;
    \item \label{D-4} $I-(J\cup K)=(I-J)\cap (I-K)$;
    \item \label{D-5} $(I-J)-K=I-(J+K)$;
    \item \label{D-6} $I+J\subseteq K$ if and only if $I\subseteq K-J$.
\end{enumerate}

Observe that $S-\mathbb{N}= \operatorname{c}(S)+\mathbb{N}=\operatorname{C}(S)$, the \emph{conductor ideal} of $S$. Similarly, for every ideal $I$ of $S$, 
\begin{equation}\label{eq:conductor-quotient}
I-\mathbb{N}=\operatorname{C}(I).
\end{equation}

Notice that for every ideals $I$ and $J$, from the definition of $J-I$, we have $I+(J-I)\subseteq J$, and thus by~\ref{D-6}, we obtain
\begin{enumerate}[({D }7)]
    \item \label{D-7} $I\subseteq J-(J-I)$.
\end{enumerate}
If we apply~\ref{D-2} to~\ref{D-7}, we deduce that $J-I\supseteq J-(J-(J-I))$, and if we use~\ref{D-7} for $J-I$, we obtain the other inclusion. This proves 
\begin{enumerate}[({D }8)]
    \item \label{D-8} $J-I=J-(J-(J-I))$.
\end{enumerate}
In particular,
\[
S-I=S-(S-(S-I)).
\]

Given an ideal $I$ of a numerical semigroup $S$, its set of \emph{pseudo-Frobenius numbers} is the set 
\[\operatorname{PF}(I)=(I-M)\setminus I,\] 
where $M$ is the maximal ideal of $S$. Notice that $\operatorname{PF}(I)$ is the set of maximal elements of $\mathbb{Z}\setminus I$ with respect to $\le_S$. The  \emph{type} of $I$ is defined as $\operatorname{t}(I)=|\operatorname{PF}(I)|$. 

According to \cite[Proposition~24]{ns-app} and \cite[Corollary~25]{ns-app}, every ideal $I $ of a numerical semigroup $S$ can expressed as the intersection of $\operatorname{t}(I)$ $\mathbb{Z}$-irreducible ideals (in the sense that they cannot be expressed as the intersection of two other ideals properly containing them). The $\mathbb{Z}$-irreducible ideals turn out to be translations of the canonical ideal of $S$. This decomposition into irreducible ideals is unique (and cannot be refined). 


We see next that this bound applies for every expression of any ideal as an intersection of other ideals properly containing it and so that none of them can be removed. But first we will prove a couple of technical lemmas, which are the analogues for ideals to \cite[Lemma~6]{ns-app} and \cite[Proposition~189]{ns-app}, respectively.

\begin{lemma}\label{lem:max-difference}
Let $S$ be a numerical semigroup. If $I$ and $J$ are ideals of $S$ with $I\subsetneq J$, then $\max(J\setminus I)\in \operatorname{PF}(I)$.
\end{lemma}
\begin{proof}
    Let $f=\max(J\setminus I)$. Then, $f+s\in J$ for all $s\in S\setminus\{0\}$,  and, as $f+s>f$, we have  $f+s\in I$, which means that $f\in \operatorname{PF}(I)$.
\end{proof}

\begin{lemma}\label{lem:char-int-ideals-gen}
    Let $I$ be an ideal of a numerical semigroup $S$, and let $I_1,\dots,I_k$ be ideals of $S$ with $I\subseteq I_i$ for all $i\in \{1,\dots,k\}$. The following are equivalent:
    \begin{itemize}
        \item $I=I_1\cap \dots \cap I_k$,
        \item $\operatorname{PF}(I)\subseteq\bigcup_{i=1}^k (\mathbb{Z}\setminus I_i)$. 
    \end{itemize}
\end{lemma}
\begin{proof}
    We are going to follow the same argument as \cite[Proposition~18]{ns-app}.

    Suppose that $I=I_1\cap \dots \cap I_k$.  Take $f\in \operatorname{PF}(I)$. Then, $f\not\in I$, and so there exists $j\in \{1,\dots,k\}$ such that $f\not\in I_j$, and thus $f\in \mathbb{Z}\setminus I_j$. 
    
    For the other implication, suppose that $I\subsetneq I_1\cap \dots \cap I_k$. Set $f=\max((I_1\cap \dots \cap I_k)\setminus I)$. By Lemma~\ref{lem:max-difference}, we get $f\in \operatorname{PF}(I)$, and by hypothesis, there exists $j\in \{1,\dots,k\}$ such that $f\in \mathbb{Z}\setminus I_j$. Thus, $f\not\in I_j$ and consequently $f\not\in I_1\cap \dots \cap I_k$, a contradiction.
\end{proof}

\begin{proposition}\label{prop:bound-decomposition-intersection}
Let $I$ be an ideal of a numerical semigroup $S$, and let $I_1,\dots,I_k$ be ideals of $S$ with $I\subseteq I_i$ for all $i\in \{1,\dots,k\}$. Suppose that $I=I_1\cap \dots \cap I_k$ and that $I\neq \bigcap_{i\in J} I_j$ for any $J\subsetneq \{1,\dots,k\}$. Then, $k
 \le\operatorname{t}(I)$.
\end{proposition}
\begin{proof}
By Lemma~\ref{lem:char-int-ideals-gen}, we know that $\operatorname{PF}(I)\subseteq \bigcup_{i=1}^k (\mathbb{Z}\setminus I_i)$. The condition $I\neq \bigcap_{i\in J} I_j$ for any $J\subsetneq \{1,\dots,k\}$, implies by the same lemma that $\operatorname{PF}(I)\not\subseteq \bigcup_{j\in J} (\mathbb{Z}\setminus I_j)$. This forces $k$ to be at most the cardinality of $\operatorname{PF}(I)$.
\end{proof}

\section{Reflexive ideals and reflexive closure}\label{sed:reflexive-ideals}


In this section, we start with our study of reflexive ideals of numerical semigroups. Let $S$ be a numerical semigroup; an ideal $I$ of $S$
is said to be \emph{reflexive} if 
$$
S-(S-I)=I.
$$
Reflexive ideals are also known in the literature as bi-dual ideals (for instance, in \cite[Section~3.2]{ns-app}) or divisorial ideals \cite{Spi15}.
For every ideal $I$ of $S$, we will denote by 
\[
I^v=S-(S-I),
\]
and call $I^v$ the \emph{reflexive closure} of $I$. 

\begin{proposition}\label{prop: star basic properties}
The map
$^v\colon \mathfrak{I}(S)\to \mathfrak{I}(S),\ I\mapsto I^v$
is well defined and fulfils the following properties, for every $I,J\in\mathfrak{I}(S)$ and $a\in \mathbb{Z}$. 
\begin{enumerate}
    \item if $I\subseteq J$, then $I^v\subseteq J^v$; 
    \item $I\subseteq I^v$;
    \item $(I^v)^v = I^v$; 
     \item $I^v\cup J^v\subseteq (I\cup J)^v$;  
    \item $(I\cap J)^v \subseteq I^v\cap J^v$;
    \item $(a+I)^v=a+I^v$.
\end{enumerate}
\end{proposition}
\begin{proof}
 Property (1) follows from  Property~\ref{D-2} of the previous section, dualizing $I \subseteq J$ with $S$ twice, and property (5) is a consequence of (1). Properties (2) and (3) are direct applications of~\ref{D-7} and~\ref{D-8}, respectively. Concerning (4), note that, applying~\ref{D-4} and~\ref{D-3},
 \[
 S-(S-(I \cup J))=S-((S-I) \cap (S-J)) \supseteq  (S-(S-I)) \cup (S-(S-J))\,.
 \]
Property (6) follows from~\ref{D-1}.
\end{proof}

As a direct consequence of item (6) of the previous proposition, we have that, if $I$ is reflexive, so is $x+I$ for every integer $x$.

Proposition~\ref{prop: star basic properties} also implies that $^v$ is a star operation on $S$, also known as the $v$-operation (see \cite{Spi15}). A \textit{star operation} on $\mathfrak{I}(S)$ is a map
    $$^*\,:\,\mathfrak{I}(S) \longrightarrow \mathfrak{I}(S)$$ such that for every $I,J \in \mathfrak{I}(S)$ one has:
\begin{enumerate}[(a)]
\item $I\subseteq I^*$,
\item if $I \subseteq J$, then $I^* \subseteq J^*$;
\item $(I^*)^*=I^*$;
\item $a+I^*=(a+I)^*$ for every $a \in \mathbb{Z}$;
\item $S^*=S$.
\end{enumerate}
If $*_1$ and $*_2$ are star operations on $\mathfrak{I}(S)$, $*_1$ is \emph{coarser} than $*_2$ if $I^{*_2} \subseteq I^{*_1}$ for every $I \in \mathfrak{I}(S)$; the $v$-operation is the coarsest star operation on $\mathfrak{I}(S)$ and coincides with the identity operation if and only if $S$ is symmetric \cite{Spi20} (in particular, in a symmetric numerical semigroup, every ideal is reflexive).



\begin{example}\label{ex:S-and-N}
    Let $S$ be a numerical semigroup. Then, $S$ and $\mathbb{N}$ are ideals of $S$. From $S+S=S$, we deduce that $S\subseteq S-S$, and if $z\in S-S$, then $z+0\in S$, forcing $z\in S$. Thus, $S- S=S$, and so $S^v=S-(S-S)=S$, meaning that $S$ is a reflexive ideal of $S$.

    We already know that $S-\mathbb{N}=\operatorname{C}(S)$. Thus, by~\ref{D-1}, $S-(S-\mathbb{N})=S-(\operatorname{c}(S)+\mathbb{N})=-\operatorname{c}(S)+\operatorname{c}(S)+\mathbb{N}=\mathbb{N}$. Hence, $\mathbb{N}$ is also a reflexive ideal of $S$.
\end{example}

\begin{remark}\label{rem:char-ref}
    In light of~\ref{D-8}, an ideal $I$ is reflexive if and only if there exists an ideal $J$ such that $I=S-J$ and one can assume $J$ to be reflexive, indeed one can choose $J=S-I$. 
\end{remark}


\begin{corollary}\label{cor:int-irred}
    The intersection of two reflexive ideals is reflexive. 
\end{corollary}
\begin{proof}
    Let $S$ be a numerical semigroup, and let $I$ and $J$ be reflexive ideals of $S$. Then by \ref{D-2}, ~\ref{D-3} and~\ref{D-4} we have 
    \[
    S-(S-(I\cap J))\subseteq S-((S-I) \cup (S-J))=(S-(S-I)) \cap (S-(S-J))=I \cap J,
    \]
    which proves that $I\cap J$ is also a reflexive ideal.
\end{proof}

According to the last result and the paragraph preceding it, the intersection of principal ideals is reflexive. We next see that $I^v$ is precisely the intersection of principal ideals, and, as a consequence, every reflexive ideal is the intersection of principal ideals.
This result is inspired by the slides \cite{vb-slides} of a talk given by V. Barucci.

\begin{proposition}\label{prop:expr-I-star}

Let $S$ be a numerical semigroup and let $I$ be an ideal of $S$. Set $A=\operatorname{Minimals}_{\le_S}(S-I)$.
Then,
\[
I^v=\bigcap_{z\in A} -z+S.
\]
In particular, the cardinality of $A$ is at most the multiplicity of $S$. Moreover, if $0\in I$, then $A\subseteq S$. 

\end{proposition}
\begin{proof}
    Notice that $x\in S-(S-I)=I^v$ if and only if $x+(S-I)\subseteq S$. This holds if and only if $x+y\in S$ for all $y\in \mathbb{Z}$ such that $y\in S-I$, or equivalently, $x\in -y+S$ for all $y\in S-I$. 
    This proves that 
    \[
    I^v=\bigcap_{z\in S-I} -z+S.
    \]
    
    Let $x,x'\in \mathbb{Z}$. Then $x\le_S x'$ if and only if $x'+S\subseteq x+S$, or equivalently, $-x+S\subseteq -x'+S$. Thus, in the above expression of $I^v$, we can replace $S-I$ with $\operatorname{Minimals}_{\le_S}(A)$, that is, the minimal generators of the ideal $S-I$; its cardinality is at most the multiplicity of $S$, since, as we already mentioned above, every element in this set must be in a different congruence class modulo the multiplicity.

    Now, suppose that $0\in I$; it is straightforward that $S-I \subseteq S$, and so $A$ is contained in $S$.
\end{proof}

The following remark gives an upper bound on the number of principal ideals involved in the (non-redundant) decomposition of a reflexive ideal in terms of principal ideals.

\begin{remark}\label{rem:on-lenght-decomposition-principal-ideals}
    Notice that if $K$ is the canonical ideal of $S$ and $I$ is an ideal of $S$, then by \cite[Lemma~14]{ns-app}, $K-K=S$ and thus from \ref{D-4} and \eqref{eq:prop-canonical} we deduce
    \[
    K-(S-I)=K-((K-K)-I)=K-(K-(K+I))=K+I.
    \]
    Hence, by \ref{D-4}, $K-((S-I)+M)=(K-(S-I))-M=(K+I)-M$.  
    Thus, using \cite[(3.3) page 37]{ns-app}, we get
\begin{align*}
    |(S-I)\setminus ((S-I)+M)|&=|(K-((S-I)+M)) \setminus (K-(S-I))|\\
    &=|((K+I)-M) \setminus (K+I)|\\
    &=\operatorname{t}(K+I).
\end{align*}
Notice that $\operatorname{Minimals}_{\le_S}(S-I)=(S-I)\setminus ((S-I)+M)$.
Therefore, from the proof of Proposition~\ref{prop:expr-I-star} we deduce that a reflexive ideal $I$ can be expressed as an intersection of at most $\operatorname{t}(K+I)$ principal ideals. If we combine this with Proposition~\ref{prop:bound-decomposition-intersection}, we get that a reflexive ideal can be expressed as an intersection of at most $\min(\operatorname{t}(I),\operatorname{t}(I+K))$ ideals (so that his intersection cannot be refined).
\end{remark}

The following example shows that $\operatorname{t}(I)$ and $\operatorname{t}(I+K)$ are not one smaller than the other in general.

\begin{example}\label{ex:ti-tik-incomparable}
    Let $S=\langle 8, 10, 12, 13, 17, 19\rangle$ and let $I=\{0, 1, 4, 5, 6, 7\}+S$. Then, $I$ is reflexive and $\operatorname{t}(I+K)=7>\operatorname{t}(I)=6$. 
    For $J=\{0, 14, 15\}+S$, we have $\operatorname{t}(J)=7>\operatorname{t}(J+K)=6$.
\end{example}


The next result states that every principal ideal cannot be expressed as the intersection of principal ideals properly containing it. As a consequence of this result and Proposition~\ref{prop:expr-I-star}, we deduce that a reflexive ideal $I$ is irreducible (in the sense that it cannot be expressed as the intersection of two  reflexive ideals properly containing it) if and only if $I$ is a principal ideal.

\begin{proposition}\label{prop:p ideals}
    Let $S$ be a numerical semigroup. If $P,I_1,\dots, I_n$ are principal ideals of $S$ such that
\[
P=I_1\cap \dots \cap I_n,
\]
then $I_k=P$ for some $k\in \{1,\dots, n\}$.
\end{proposition}
\begin{proof}
    Let $z,x_1,\dots, x_n \in \mathbb{Z}$ be such that
\[
z+S=(x_1+S)\cap \dots\cap (x_n+S).
\]
 The claim is equivalent to prove that $z=x_k$ for at least one $k\in \{1,\dots, n\}$. We can write
\[
S=(x_1-z+S)\cap \dots\cap (x_n-z+S).
\]
Let $a_k=z-x_k$ for all $k=1,\dots, n$. We have that 
\[
0\in S=(-a_1+S)\cap \dots\cap (-a_n+S)\subseteq -a_k+S
\]
and thus $a_k \in S$ for all $k=1,\dots, n$, in particular, $a_k \geq\min(S)=0$. We need to show that $a_k=0$ for some $k\in \{1,\dots, n\}$.\\
By way of contradiction, assume that $a_k>0$ for all $k=1,\dots, n$, then $\operatorname{f}(S)<\operatorname{f}(S)+a_k$ and therefore $ \operatorname{f}(S)+a_k\in S$ for all $k=1,\dots, n$, which is equivalent to $\operatorname{f}(S)\in -a_k+S$ for all $k=1,\dots, n$. This implies
\[
\operatorname{f}(S)\in (-a_1+S)\cap \dots\cap (-a_n+S)=S
\]
and that is a contradiction.
\end{proof}

\begin{remark}
    With the same proof of Proposition~\ref{prop:p ideals}, it can be proven that no ideal of $S$ is equal to an intersection of translates of itself, since principal ideals are translates of $S$. 
\end{remark}

The decomposition given in Proposition~\ref{prop:expr-I-star} is in general far from being optimal as the following example shows. {The computations can be performed using the \texttt{NumericalSgps} \cite{numericalsgps} \texttt{GAP} \cite{GAP4} package.

\begin{example}
    Let $S=\langle  6, 7, 16\rangle$ and let $I=\{ 0, 1, 3, 5\}+S$. Then, \[\operatorname{Minimals}_{\le_S}(S-I)=\{13, 18, 21, 22, 23\},\] while $I=(-13+S)\cap (-18+S)\cap (-21+S)$.
\end{example}

As a consequence of Proposition~\ref{prop:expr-I-star} we obtain the following characterization of reflexive ideals.

\begin{corollary}\label{cor:char-ref}
    Let $S$ be a numerical semigroup and let $I$ be an ideal of $S$. Then $I$ is reflexive if and only if it is the intersection of principal ideals. In particular, $I$ is the intersection of all principal ideals containing $I$.
\end{corollary}

\begin{proof}
If $I$ is reflexive, then $I=I^v$, and by Proposition~\ref{prop:expr-I-star}, $I$ is the intersection of principal ideals.

We know that principal ideals are reflexive, since we have shown that $S$ is reflexive and that the translation of every reflexive ideal is again reflexive. Also, the intersection of reflexive ideals is reflexive, and thus if $I$ is the intersection of principal ideals, then $I$ is reflexive.

If $I$ is reflexive, then from the proof of Proposition~\ref{prop:expr-I-star}, we know that $I=\bigcap_{z\in S-I} -z+S$. The ideals $-z+S$, with $z\in S-I$ are precisely the principal ideals of $S$ containing $I$.
\end{proof}

The following result is also extracted from \cite{vb-slides}. We include the proof for sake of completeness (and also because our notation is slightly different).

\begin{lemma}\label{lem:I-not-included-principal-ideal}
Let $S$ be a numerical semigroup, let $I$ be an ideal of $S$, and let $z$ be an integer. Then, $I\not\subseteq z+S$ if and only if $\operatorname{f}(S)+z\in I+\operatorname{K}(S)$.
\end{lemma}
\begin{proof}
If $I\not\subseteq z+S$, then there exists $i\in I$ such that $i-z\not\in S$. By the definition of $\operatorname{K}(S)$, we have that $\operatorname{f}(S)-(i-z)\in \operatorname{K}(S)$ and so $\operatorname{f}(S)+z=i+k$ for some $k\in \operatorname{K}(S)$.

For the other implication, if $\operatorname{f}(S)+z=i+k$ for some $i\in I$ and $k\in K$, then $\operatorname{f}(S)-(i-z)\in \operatorname{K}(S)$, which by the definition of the canonical ideal of $S$ translates to $i-z\not\in S$, and this is equivalent to $i\not\in z+S$. Thus, $I\not\subseteq z+S$.
\end{proof}



By a theorem of Jäger (see, for instance \cite[Theorem~3]{ns-app}), for every ideal $I$ of $S$,
\begin{equation}\label{eq:prop-canonical}
\operatorname{K}(S)-(\operatorname{K}(S)-I)=I
\end{equation}
(and this property characterizes $\operatorname{K}(S)$ up to translations).

The next result offers an alternate way of computing $I^v$ in terms of the canonical ideal of $S$.

\begin{lemma}\label{lem:calculate-I*-with-K}
    Let $S$ be a numerical semigroup. Then, 
    \[I^v=(\operatorname{K}(S)+I)-\operatorname{K}(S).\]    
\end{lemma}
\begin{proof}
Let $K=\operatorname{K}(S)$.  Notice that
  \begin{align*}
        (K+I)-K &= (K-(K-(K+I)))-K  && \text{by \eqref{eq:prop-canonical}}\\
        &= K-((K-(K+I))+K) && \text{by \ref{D-5}}\\
        &=K-(((K-K)-I)+K) && \text{by \ref{D-5}}\\
        &=(K-K)-((K-K)-I) && \text{by \ref{D-5}}\\
        &=S-(S-I)\,.
    \end{align*}
We are using that $K-K=S$, which follows by applying \cite[Lemma~14]{ns-app} to $K$.
\end{proof}

With this it is easy to show that $I^v$ is always between $I$ and $I+\operatorname{K}(S)$.

\begin{proposition}\label{prop:bounding-star-canonical}
Let $S$ be a numerical semigroup and let $I$ be an ideal of $S$. Then,
\[
I^v\subseteq I+\operatorname{K}(S).
\]
\end{proposition}
\begin{proof}
This follows from Lemma~\ref{lem:calculate-I*-with-K} and~\ref{D-6}:
\[
I^v =I^v+0\subseteq I^v+\operatorname{K}(S) \subseteq \operatorname{K}(S)+I.\qedhere
\]
\end{proof}

Let $S$ and $T$ be numerical semigroups, with $S\subseteq T$ ($T$ is an over-semigroup of $S$). If $I$ is an ideal of $T$, then $I=I+
0\subseteq I+S\subseteq I+T=I$, implying that $I$ is an ideal of $S$. That is,
\[
S\subseteq T \text{ implies } \mathfrak{I}(T)\subseteq \mathfrak{I}(S).
\]

The next theorem states that every ideal of an over-semigroup containing the canonical ideal of $S$ is reflexive.

\begin{theorem}[{\cite[Theorem~(A)]{vb-slides}}]\label{thm:A-Valentina}
    Let $S$ be a numerical semigroup and let $T$ be an over-semigroup of $S$. The following are equivalent.
    \begin{enumerate}
        \item $\operatorname{K}(S)\subseteq T$.
        \item Every ideal of $T$ is a reflexive ideal of $S$.
    \end{enumerate}
\end{theorem}
\begin{proof}
    Suppose that $\operatorname{K}(S)\subseteq T$ and let $I$ be an ideal of $T$. By Proposition~\ref{prop:bounding-star-canonical}, $I^v\subseteq I+\operatorname{K}(S)\subseteq I+T=I$, which means that $I$ is reflexive.

    The other implication requires a bit more effort. Suppose on the contrary that $\operatorname{K}(S)\not\subseteq T$, and let $U=T\cup (x+1+\mathbb{N})$, with $x\in \operatorname{K}(S)\setminus T$; thus, $\operatorname{f}(S)-x\not\in S$. Notice that $U$ is the union of two ideals of $T$, and so it is an ideal of $S$. We show that $U$ is not reflexive. 
    
    Suppose to the contrary that $U$ is reflexive. Then by Corollary~\ref{cor:char-ref}, $U$ is the intersection of all principal ideals containing $U$. Let $z$ be an integer such that $U\subseteq z+S$. Then, $z+\operatorname{f}(S)=\operatorname{f}(z+S)\le \operatorname{f}(U)=x$. If $x= z+\operatorname{f}(S)$, then $0=z+\operatorname{f}(S)-x\in z+S$ (since $0\in U\subseteq z+S$). But then $\operatorname{f}(S)-x\in S$ in contradiction with the choice of $x$. Thus, $z+\operatorname{f}(S)<x$, and so $x\in z+S$. This, in particular, implies that $x$ is in the intersection of all principal ideals containing $U$ and consequently in $U^v$; whence, $U$ is not reflexive.
\end{proof}

This last result can be seen as a generalization of the fact that if $S$ is symmetric, then all its ideals are reflexive. Notice that there might be reflexive ideals of $S$ that are not ideals of an over-semigroup of $S$ containing its canonical ideal. 

\begin{example}
    Let $S=\langle 5,11,19,23 \rangle$ and let $K$ be its canonical ideal. The Frobenius number of $S$ is $18$, and $18-1=17\not\in S$. This means that $1\in K$, and so every over-semigroup containing $K$ must contain $1$, which forces this over-semigroup to be $\mathbb{N}$. However, $I=\{ 0, 14, 17, 18 \}+S$ is reflexive and clearly it is not an ideal of $\mathbb{N}$.
\end{example}


If $J$ and $I$ are ideals of a numerical semigroup $S$, we say that $I$ is $J$-{\emph reflexive} if $J-(J-I)=I$. The following result can be understood as a sort of transitiveness of the property of being reflexive.

\begin{lemma}\label{lemma: reflexivity}
Let $S$ be a numerical semigroup and let $I,J$ be ideals of $S$. If $J-(J-I)=I$ and $S-(S-J)=J$, then $S-(S-I)=I$.
\end{lemma}
\begin{proof}
Notice that $(S-J)+(J-I)\subseteq S-I$, since if $t \in S-J$, then $t+J \subseteq S$ and  by~\ref{D-1} and~\ref{D-2}, $t+(J-I)=(t+J)-I \subseteq S-I$. We have that
\begin{align*}
    I &=J-(J-I) && \text{(hypothesis)}\\
    &=(S-(S-J))-(J-I) && \text{(hypothesis)}\\
    &=S-((S-J)+(J-I)) && \text{(by }\ref{D-5}\text{)}\\
    & \supseteq S-(S-I),  && \text{(by }\ref{D-2}\text{)}
\end{align*}
which is equivalent to the reflexivity of $I$.
\end{proof}


Reflexive ideals of a numerical semigroup sharing the same conductor as the ambient semigroup are always integral, as wee see next. We will make use of this fact several times later.

\begin{lemma}\label{lem:I tilde in S}
Let $S$ be a numerical semigroup. If $I$ is a reflexive ideal of $S$ such that $\operatorname{c}(I)=\operatorname{c}(S)$, then $I \subseteq S$.
\end{lemma}
\begin{proof}
    Set $c=\operatorname{c}(S)$. As $\operatorname{c}(I)=c$, we have  $\operatorname{f}(S)=\operatorname{f}(I)$.
    
    Since $c+\mathbb{N}\subseteq I$, by~\ref{D-2} we have  $S-I \subseteq S-(c+\mathbb{N})=\mathbb{N}$ (by equation \eqref{eq:conductor-quotient} of Section 2); therefore $\min(S-I) \geq 0$. Observe that $I \subseteq S$ if and only if $0 \in S-I$. Suppose that $\min(S-I)>0$. Then,  $\operatorname{f}(S)+(S-I) \subseteq S$ which is equivalent to $\operatorname{f}(S) \in S-(S-I)=I$, which is a contradiction. It follows that $\min(S-I)=0$ and thus $I \subseteq S$.
\end{proof}

\begin{remark}\label{rem:K vs S}
In \cite[Proposition 20]{ns-app} it is shown that $\operatorname{K}(S)$ is the largest relative ideal of $S$ among those with conductor $\operatorname{c}(S)$. Using this fact, it is possible to prove that $\operatorname{K}(S)$
is the only irreducible relative ideal, up to translations (see 
\cite[Proposition 24]{ns-app}).

The previous lemma shows that,
if we restrict ourselves to the set of reflexive ideals, sharing the same conductor $\operatorname{c}(S)$, the semigroup $S$ itself is the largest. This is another way to 
see that $S$ is irreducible in the set of reflexive ideals, and in fact it it the only irreducible, up to translations
(cf. Proposition \ref{prop:p ideals}). So in the set of reflexive ideals $S$ plays the same role as $\operatorname K(S)$
in the set of all relative ideals. Note that, when $S$ is symmetric, all relative ideal are reflexive and, in fact, $\operatorname K(S)=S$.
\end{remark}


Let $S$ be a numerical semigroup and let $I$ be an ideal of $S$. We say that $I$ is \emph{normalized} if $\min(I)=0$. Notice that if $0\in I$, then $I+S=I$, implies that $S=0+S\subseteq I+S=I$. Thus, $I$ is normalized if and only if  $S\subseteq I\subseteq \mathbb{N}$. 


The set of normalized ideals of $S$ is denoted by $\Ni[S]$. The sum of two normalized ideals if again a normalized ideal, and so $(\Ni,+)$ is a monoid, which is isomorphic to the ideal class monoid of $S$ \cite{b-k, icm-apery}. This monoid completely determines $S$ (see \cite{isom}). 

\begin{remark}\label{rem:oversemigroup-is-ideal}
Notice that if $T$ is an over-semigroup of $S$, then $T\subseteq T+0\subseteq T+S\subseteq T+T=T$, and so $T+S=T$, meaning that $T$ is a normalized ideal of $S$. Actually, over-semigroups of $S$ are precisely those normalized ideals of $S$ that are idempotents (with respect to $+$).
\end{remark}

\begin{remark}\label{rem:I-I-oversemigroup}
If $I$ is an ideal of a numerical semigroup $S$, and $x,y\in \mathbb{Z}$ are such that $x+I\subseteq I$ and $y+I\subseteq I$, then $(x+y)+I=x+
(y+I)\subseteq x+I\subseteq I$, which means that $I-I$ is closed under addition. Also, $0=\min(I-I)$, and $\operatorname{c}(I)+\mathbb{N}\subseteq I-I$, which means that $I-I$ is itself a semigroup. As $S+I=I$, we get $S\subseteq I-I$, which results in $I-I$ being an over-semigroup of $S$ (and thus a normalized ideal of $S$).
\end{remark}

\begin{remark}\label{rem:S-T-ideal-T}
    Let $S$ be a numerical semigroup, and let $T$ be an over-semigroup of $S$. Then $S-T$ is an ideal of $T$. In fact, given $x\in S-T$ and $t\in T$, $x+t\in S-T$ if and only if for all $t'\in T$, $x+t+t'\in S$. But this follows from the fact that $t+t'\in T$ and $x\in S-T$.
\end{remark}







Observe that if $I\in \Ni[S]$, then $S\subseteq I\subseteq \mathbb{N}$. Hence, by Proposition~\ref{prop: star basic properties}, $S^v\subseteq I^v\subseteq \mathbb{N}^v$. As $S^v=S$ and $\mathbb{N}^v=\mathbb{N}$, we deduce that $S\subseteq I^v\subseteq \mathbb{N}$, which leads to $I^v\in \Ni[S]$. 

Notice also that the reflexive property is closed under translations. Thus, if we determine the set of normalized reflexive ideals, by translation, we determine the set of reflexive ideals of a numerical semigroup. 

In order to determine all normalized reflexive ideals, we will use a particular class of reflexive ideals, which we study in the next section.


\section{Integrally closed ideals}\label{sec:int-closed-ideals}

Let $S$ be a numerical semigroup and let $s\in S$. The set 
\[
S(s)=\{t\in S : t\ge s\}
\]
is an ideal of $S$. We say that an ideal $I$ of $S$ is \emph{integrally closed} if $I=S(s)$ for some $s\in S$.

Integrally closed ideals carry significant data about their ambient numerical semigroup. However, it is worth noting that these data can be accessed without studying the entire family of such ideals: in fact, two kinds can be distinguished. The ideals of the form $S(s)$, for some $s > \operatorname{c}(S)$, are all translations of $\operatorname{C}(S)$; for this reason, they are ``trivial'', in the sense that many of the properties that we study trivially hold for this family of ideals. The remaining ideals are much richer in structure and cannot be obtained as a translation of any other integrally closed ideals (see Remark~\ref{rem:i=j}). With this observation, we can reduce the study of integrally closed ideals to a finite number of them, which could be convenient also from a computational view point. In what follows, this distinction plays a role only in Remarks~\ref{rem:antichains} and \ref{rem:int closed -> trace}, and in Proposition~\ref{prop:phi is inj}.

Observe that $S(s)=S\cap(s+\mathbb{N})$.
As both $S$ and $s+\mathbb{N}$ are reflexive (Example~\ref{ex:S-and-N}), and the intersection of reflexive ideals is reflexive, we deduce that $S(s)$ is reflexive. Therefore, we also have the following result.

\begin{lemma}\label{lem:ic-implies-ref}
    Every translation of an integrally closed ideal of a numerical semigroup is reflexive.
\end{lemma}

The next lemma is the analogue to \cite[Lemma 3.3]{CDS} for numerical semigroups. Recall that for an ideal $I$ of a numerical semigroup $S$, $\widetilde{I}=-\operatorname{c}(I)+\operatorname{c}(S)+I$. If, in addition, $I \in \mathfrak{I}_0(S)$ and $\widetilde I \subseteq S$, denote by $\bar{I}$ the unique integrally closed ideal of $S$ such that $\min(\widetilde{I})=\min(\bar{I})$, that is, 
\[
\bar{I}=S(\min(\widetilde{I})).
\]
Note that Lemma \ref{lem:I tilde in S} ensures that if $I$ is reflexive, then $\widetilde I \subseteq S$.

\begin{lemma}\label{lem:s-i=bar i-i}
    If $I \in \mathfrak{I}_0(S)$ and $\widetilde I \subseteq S$, then $S-\widetilde{I}=J-\widetilde{I}$ and $S-I=J-I$ for every integrally closed ideal $J$ containing $\bar{I}$.
\end{lemma}
\begin{proof}
First, notice that by~\ref{D-1}, if $S-\widetilde{I}=J-\widetilde{I}$, then $S-I=J-I$. So we focus on proving that $S-\widetilde{I}=J-\widetilde{I}$ holds for every integrally closed ideal containing $\bar{I}$.

The inclusion $S-\widetilde{I} \supseteq J-\widetilde{I}$ is clear by~\ref{D-2}, since $J\subseteq S$. For the other inclusion, notice that by definition, $\widetilde{I}+(S-\widetilde{I}) \subseteq S$. From $\operatorname{c}(S)+\mathbb{N}\subseteq \widetilde{I}$, by using~\ref{D-2} we obtain $S-\widetilde{I} \subseteq S-\operatorname{C}(S)=\mathbb{N}$ (by~\eqref{eq:conductor-quotient}), which forces $\min(S-\widetilde{I})\ge 0$. Thus, $\min(\widetilde{I})+\min(S-\widetilde{I}) \geq \min(\widetilde{I})=\min(\bar{I})\ge \min(J)$, which together with $\widetilde{I}+(S-\widetilde{I}) \subseteq S$ and the fact that $J$ is integrally closed, implies that  $\widetilde{I}+(S-\widetilde{I}) \subseteq J$. Hence by~\ref{D-6}, $(S-\widetilde{I}) \subseteq J-\widetilde{I}$, and this concludes the proof.
\end{proof}

Recall that if $J$ and $I$ are two ideals of $S$, we  say that $I$ is $J$-{\emph reflexive} if $J-(J-I)=I$. With this terminology, 
the following result is a translation of \cite[Proposition~4.5]{CDS} to numerical semigroups. 
\begin{proposition}\label{prop: J-reflexivity}
    Let $S$ be a numerical semigroup, and let $I$ be an ideal of $S$. The following conditions are equivalent: 
    \begin{itemize}
        \item $I$ is reflexive,
        \item $\widetilde{I} \subseteq S$ and $I$ is $J$-reflexive for every integrally closed ideal $J$ such that $\widetilde{I} \subseteq J$.
    \end{itemize} 
    In particular, if $I$ is a non-principal reflexive ideal of $S$ we have that $\widetilde{I} \subseteq M$, and thus $I$ is $M$-reflexive, where $M$ is the maximal ideal of $S$.
\end{proposition}
\begin{proof}
    Suppose that $I$ is reflexive. By Lemma~\ref{lem:I tilde in S} we have $\widetilde{I} \subseteq S$ and therefore we can apply Lemma~\ref{lem:s-i=bar i-i}, 
    which yields $S-I=J-I$, so $J-(J-I) = J-(S-I) \subseteq S-(S-I)=I$.

    Now, fix an integrally closed ideal $J$ such that $\widetilde{I} \subseteq J$. Assume that $I$ is $J$-reflexive and notice that by Lemma~\ref{lem:ic-implies-ref} we have that $J$ is reflexive; explicitly, one has $J-(J-I)=I$ and $S-(S-J)=J$, and thus the reflexivity of $I$ follows from Lemma~\ref{lemma: reflexivity}. 

    Finally, let us prove the second part of the statement. By \cite[Proposition~20]{ns-app}, if $I$ is an ideal of $S$, then $\widetilde{I}\subseteq \operatorname{K}(S)$. 
    In particular, $\min(\widetilde{I})\ge 0$. Notice that if this inequality is an equality, then $S\subseteq \widetilde{I}$. If, in addition, $I$ is reflexive, then by Lemma~\ref{lem:I tilde in S}, $\widetilde{I}\subseteq S$, and so $S=\widetilde{I}$. Thus, if $I$ is reflexive and $\min(\widetilde{I})=0$, then $I$ must be a principal ideal. Since we are assuming that $I$ is not principal, $\min{\widetilde{I}}>0$ and so $\widetilde{I}\subseteq M$. As $M=S(\min(M))=S(\operatorname{m}(S))$ is integrally closed, we conclude that $I$ is $M$-reflexive.
\end{proof}


Recall that if $I$ is an ideal of $S$, then $I-I$ is an over-semigroup of $S$. We see next that $I-I$ is also an over-semigroup of $J-J$ for every integrally closed ideal $J$ of $S$ containing $\bar{I}$. 

\begin{proposition}\label{prop:s-s in e-e}
    If $I \in \mathfrak{I}_0(S)$ is reflexive, then $J-J\subseteq I-I$ for every integrally closed ideal $J$ such that $ \bar{I} \subseteq J$.
\end{proposition}
\begin{proof} 
We observed that $I$ reflexive implies that $\bar I$ is well-defined. Let $t \in J-J $ and $x \in \widetilde{I}$. Since $S-\widetilde{I}=J-\widetilde{I}$ by Lemma~\ref{lem:s-i=bar i-i}, for every $u \in S-\widetilde{I}$, we have $t+x+u \in t+J \subseteq J \subseteq S$, and then $t+x \in S-(S-\widetilde{I})=\widetilde{I}$, which implies $t \in \widetilde{I}-\widetilde{I}=I-I$ (this equality holds by ~\ref{D-1}).
\end{proof}

With this, we can offer a numerical semigroup version of \cite[Proposition 3.5]{CDS}, which will become useful later.

\begin{corollary}\label{cor:m+bar I-bar I}
    If $I \in \mathfrak{I}_0(S)$ is reflexive, then $\min(\widetilde{I})+\left(\bar{I}-\bar{I}\right) \subseteq \widetilde{I} \subseteq \bar{I}$.
\end{corollary}
\begin{proof}
    The inclusion $\widetilde{I} \subseteq \bar{I}$ follows from Lemma~\ref{lem:I tilde in S}. For the other inclusion, let $t \in \bar{I}-\bar{I}$. By Proposition~\ref{prop:s-s in e-e} we have $\bar{I}-\bar{I} \subseteq I-I=\widetilde{I}-\widetilde{I}$, therefore $t + \widetilde{I} \subseteq \widetilde{I}$ and thus $\widetilde{I}$ is an ideal of $\bar{I}-\bar{I}$, then $\min(\widetilde{I})+(\bar{I}-\bar{I}) \subseteq \widetilde{I}$.
\end{proof}

\section{irreducible normalized reflexive ideals}\label{sec:irred-ref-ideals}


Let $S$ be a numerical semigroup and let $I$ be a normalized reflexive ideal of $S$. We say that $I$ is \emph{NR-irreducible} if $I$ cannot be expressed as the intersection of two normalized reflexive ideals properly containing $I$. As the intersection of reflexive ideals is reflexive, every reflexive ideal is the intersection of NR-irreducible ideals.

\begin{lemma}\label{lem:ic-r-irred}
    Let $S$ be a numerical semigroup and let $I \in \mathfrak{I}_0(S)$. If $\widetilde{I}$ is integrally closed, then $I$ is NR-irredudible.
\end{lemma}
\begin{proof}
    By Lemma~\ref{lem:ic-implies-ref}, we know that $\widetilde{I}$ is reflexive.
    Let $s \in \mathbb{Z}$ such that $s+I=\widetilde{I}=\bar{I}$ and notice that $s=\min(\widetilde{I})=\min(\bar{I})$. Suppose that there are reflexive ideals $J_1,J_2 \in \mathfrak{I}_0(S)$ such that $I \subsetneq J_1,\,I \subsetneq J_2$ and $I=J_1 \cap J_2$; then $\widetilde{I}=(s+J_1) \cap (s+J_2)$. Fix $i \in \{1,2\}$. We have $\widetilde{I} \subsetneq s+J_i$ and therefore $\operatorname{c}(s+J_i) \leq \operatorname{c}(S)$. Since $s+J_i$ is reflexive, if $\operatorname{c}(s+J_i)=\operatorname{c}(S)$, Lemma~\ref{lem:I tilde in S} would give $s+J_i \subseteq S$, but this in turn would imply $s+J_i \subseteq \widetilde{I}$, which is a contradiction. Thus, it must be $\operatorname{c}(s+J_i) < \operatorname{c}(S)$ for every $i=1,2$ and therefore $\operatorname{f}(S) \in (s+J_1) \cap (s+J_2) =\widetilde{I}$, which raises another contradiction. It follows that $I$ is NR-irredudible.
\end{proof}

Recall that for $s\in S$, we defined $S(s)=\{ t\in S : t\ge s\}$. Proposition~\ref{prop:expr-I-star} can be further particularized as follows.

\begin{proposition}\label{prop:exp-ref}
    Let $S$ be a numerical semigroup with $M=\operatorname{M}(S)$ and let $I$ be a normalized reflexive ideal of $S$. Then, there exists an anti-chain with respect to the order induced by ${M-M}$, $s_1,\dots, s_n\in S$, such that 
    \[
    I=(-s_1+S(s_1))\cap \dots \cap (-s_n+S(s_n)).
    \]
\end{proposition}
\begin{proof}
As $I$ is a normalized reflexive ideal of a numerical semigroup $S$, we have $I=I^v$. By Proposition~\ref{prop:expr-I-star}, there exists $s_1,\dots,s_k\in S$, incomparable with respect to $S$ (they are precisely the minimal generators of the ideal $S-I$), such that 
\[
I=(-s_1+S)\cap \dots \cap (-s_n+S).
\]

Observe also that as $I\subseteq \mathbb{N}$, we have
\begin{align*}
I& =(-s_1+S)\cap \dots \cap (-s_n+S)\cap \mathbb{N}\\
& = ((-s_1+S)\cap\mathbb{N})\cap \dots \cap ((-s_n+S)\cap\mathbb{N})\\
& = (-s_1+S(s_1))\cap \dots \cap (-s_n+S(s_n)).
\end{align*}
Take $i$ and $j$ such that $s_i<s_j$. Notice that $-s_i+S(s_i)\subseteq -s_j+S(s_j)$ if and only if $s_j-s_i+S(s_i)\subseteq S(s_j)$.  

Suppose that $s_j-s_i+S(s_i)\subseteq S(s_j)$. From $s_j>s_i$, we deduce that $S(s_j)\subseteq S(s_i)$, and consequently $s_j-s_i+S(s_i)\subseteq S(s_i)$. Thus, $s_j-s_i\in S(s_i)-S(s_i)$.

Now, suppose that $s_j-s_i\in S(s_i)-S(s_i)$. Then, $s_j-s_i+x\in S(s_i)$ for all $x\in S(s_i)$. But then $x-s_i\ge 0$ and so $x+s_j-s_i\ge s_j$, which leads to $x+s_j-s_i\in S(s_j)$. This shows that $s_j-s_i+S(s_i)\subseteq S(s_j)$.

This proves that if $s_i<s_j$, then $-s_i+S(s_i)\subseteq -s_j+S(s_j)$ if and only if $s_j-s_i\in S(s_i)-S(s_i)$. 

Finally, suppose that $s_i\le_{M-M} s_j$, with $i\neq j$. Then, $s_j-s_i\in M-M$ and as $M-M\subseteq \mathbb{N}$, we have in addition that $s_j>s_i$. By Lemma~\ref{lem:ic-implies-ref} and Proposition~\ref{prop:s-s in e-e}, we deduce that $M-M\subseteq S(s_i)-S(s_i)$, and so $s_j-s_i\in S(s_i)-S(s_i)$, which by the above discussion, means that $-s_i+S(s_i)\subseteq -s_j+S(s_j)$. Thus, we can remove $-s_j+S(s_j)$ from the expression of $I$.
\end{proof}


A decomposition of a reflexive ideal $I$ in the form given in Proposition~\ref{prop:exp-ref}, where none of the factors can be removed, will be called an \emph{R-factorization}, and the number of factors will be its \emph{length}.

By Remark~\ref{rem:on-lenght-decomposition-principal-ideals}, the length of a factorization of an ideal $I$ is at most $\min(\operatorname{t}(I),\operatorname{t}(I+K))$. It may happen that this bound is far from being reached.
\begin{example}
    Let us go back to Example~\ref{ex:ti-tik-incomparable}, with $S=\langle 8, 10, 12, 13, 17, 19\rangle$ and $I=\{0, 1, 4, 5, 6, 7\}+S=-12+S(12)$ (and thus it is a translation of an integrally closed ideal). Then, $I$ is reflexive and $\min(\operatorname{t}(I+K),\operatorname{t}(I))=6$. 
\end{example}

The next example shows that the above bound can be attained (and thus it is sharp).

\begin{example}
Consider the numerical semigroup $S=\langle 6,10,11,13,14\rangle$. The ideal $I=\{0,1,8,9\}+S$ decomposes as
\[
I=(-10+S(10))\cap (-11+S(11))\cap (-12+S(12))\cap (-13+S(13))
\]
and $\operatorname{t}(I)=\operatorname{t}(\operatorname{K}(S)+I)=4$. This decomposition is unique.
\end{example}

\begin{lemma}\label{lem:I-M}
    Let $M=\operatorname{M}(S)$. If $I\in \mathfrak{I}_0(S)$, then 
    \[
    I-M=I\cup\operatorname{PF}(I).
    \]
\end{lemma}
\begin{proof}
    Since $I$ is an ideal we get $I+M\subseteq I$, while from the definition of pseudo-Frobenius numbers we have $\operatorname{PF}(I)+M\subseteq I$, therefore $I\cup\operatorname{PF}(I)\subseteq I-M$ by \ref{D-6}.\\
    For the other inclusion, let $x\in I-M$, then either $x\in
    I$ or $x\in (I-M)\setminus I=\operatorname{PF}(I)$. In conclusion, $I-M=I\cup \operatorname{PF}(I)$.
\end{proof}

In particular, by applying this result to $M$ and $S$, respectively, we obtain $M-M=M\cup \operatorname{PF}(M)$ and $S-M=S\cup\operatorname{PF}(S)$. As $M-M=S-M$, we deduce that 
\begin{enumerate}
    \item $M-M$ is reflexive by Remark~\ref{rem:char-ref}, and 
    \item $M-M=S\cup \operatorname{PF}(S)$.
\end{enumerate}

\begin{proposition}\label{prop:min ref}
Let $S$ be a numerical semigroup different than $\mathbb{N}$ and let $M=\operatorname{M}(S)$. The normalized reflexive ideal $M-M$
    is minimal with respect to the set-theoretic inclusion among the normalized reflexive ideal of $S$ that are different than $S$. 
\end{proposition}

\begin{proof}
    Let $I\in\mathfrak{I}_0(S)\setminus\{S\}$. We have that $S\subseteq I$, because $I$ is normalized, and therefore $S-I\subseteq S-S=S$ by \ref{D-2}. Moreover, since $I\neq S$, we get $0\not\in S-I$ (otherwise, $I\subseteq S$ and thus $I=S$) and this implies that $S-I\subseteq M$. Thus we have
    \[
    (S-I)+(M-M)\subseteq M+(M-M) \subseteq M\subseteq S
    \]
    and it follows that $M-M\subseteq S-(S-I)$ by \ref{D-6}. In particular, $M-M$ is contained in every normalized reflexive ideals which is different than $S$. 
\end{proof}

\begin{remark}
    The argument in the proof of Proposition~\ref{prop:min ref} can be generalized from numerical semigroup rings to any commutative, noetherian, one-dimensional local ring $(R,\mathfrak{m},k)$, such that the residue field $k$ is infinite and the integral closure $\overline{R}$ of $R$ in the total ring of fraction $Q(R)$ is of finite type as $R$-module (these assumption are the setup of \cite{CDS}). As noted in the introduction, for any $I$ and $J$ fractional ideals of $R$ in $Q(R)$, we denote the colon ideal as
    \[
    I:J=\{x\in Q(R):xJ\subseteq I\}
    \]
    and $I$ is reflexive if $R:(R:I)=I$.\\
    Any regular fractional ideal $I$ of $R$ is isomorphic to a fractional ideal between $R$ and $\overline R$, since if  $x\in I$ is a principal reduction of $I$, then we can divide by $x$ and get $1\in \frac{1}{x}I$ (for more details on principal reductions of fractional ideals, see \cite{CDS}). Let $I$ be a fractional reflexive ideal of $R$ such that $R\subsetneq I \subseteq \overline R$; then $R\subseteq I$ and thus $R:I \subseteq R$. If $ 1\in R: I$ then $R=R: I$, i.e. $I=R$, a contradiction. Thus, $R: I\subseteq \mathfrak{m}$. With the same computation of the proof of Proposition~\ref{prop:min ref} we get
    \[
    \mathfrak{m}:\mathfrak{m} \subseteq R:(R:I)=I
    \]
    for every fractional reflexive ideal $I$ of $R$ such that $R\subsetneq I \subseteq \overline R$.
\end{remark}

We see next that to obtain an R-factorization, it is not enough to eliminate those $-s_j+S(s_j)$ containing another single factor, as it can happen that $-s_j+S(s_j)$ contains an intersection of more than one factor, all of them not contained in $-s_j+S(s_j)$. Moreover, R-factorizations do not need to be unique nor have the same length.

\begin{example} \label{ex:8-9-10}
Let us consider $S=\langle 8, 9, 10, 11, 12, 15\rangle$, and let $I=\{ 0, 7, 13, 14 \}+S$.
\begin{verbatim}
gap> S:=NumericalSemigroup(8, 9, 10, 11, 12, 15);;
gap> I:=[ 0, 7, 13, 14 ]+S;;
gap> M:=MaximalIdeal(S);;
gap> MinimalGenerators(S-I);
[ 8, 9, 10, 11, 12, 15 ]
gap> SmallElements(M-M);
[ 0, 7 ]
\end{verbatim}

Notice that $8\le_{M-M} 15$, and so $15$ is not needed for the representation of $I$ as intersection of translations of integrally closed ideals. 

It can be shown that 
\begin{align*}
    I & =(-8+S(8))\cap (-10+S(10))\cap (-12+S(12)) \\
    & = (-8+S(8))\cap (-9+S(9))\cap (-11+S(11))\cap (-12+S(12)),
\end{align*}
and that none of these decompositions can be refined. Moreover, $-10+S(10)$ does not include any of the $-k+S(k)$ for $k\in \{8,9,11,12\}$. Thus, $I$ admits two R-factorizations, one of length three and another of length four. Note that in this case $\operatorname{t}(I)$.
\end{example}

\begin{remark}\label{rem:antichains}
    Notice that if $I$ is a normalized ideal of a numerical semigroup $S$, then $I=\{0,g_1,\dots,g_n\}+S$, where $g_1,\dots,g_n$ are incomparable gaps (with respect to $\le_S$) of $S$ (actually, $\{0,g_1,\dots,g_n\}=\operatorname{Minimals}_{\le_S}(I)$). Thus, the set of normalized ideals of $S$ is in one-to-one correspondence with the set of anti-chains of gaps of $S$ with respect to $\le_S$ (this fact was considered in \cite{b-k,icm-apery} to give bounds for the cardinality of $\Ni$). By looking at Proposition~\ref{prop:exp-ref} we may think that we have a correspondence between anti-chains of elements in $S$ smaller than or equal to the conductor and reflexive ideals (if $k\ge \operatorname{c}(S)$, then $-k+S(k)=\mathbb{N}=-\operatorname{c}(S)+S(\operatorname{c}(S))$).
    However, this correspondence is not as nice as the one for anti-chains of gaps and normalized ideals, since it may happen that two different anti-chains produce the same reflexive ideal, as it happens in Example~\ref{ex:8-9-10}.

Proposition~\ref{prop:exp-ref} provides a procedure to compute the set of normalized reflexive ideals of a numerical semigroup $S$. We start by computing $M-M$, with $M$ the maximal ideal of $S$. Then, we compute the set of all anti-chains of small elements of $S$ with respect to the order induced by $M-M$. For each anti-chain $A$, we compute $\bigcap_{x\in A} -x+S(x)$. This procedure is faster than computing the set of all normalized ideals and then filter those that are reflexive. 
\end{remark}






We might pose the same question, but now for integral ideals (ideals included in their ambient semigroup). In this setting, irreducible means that it is not the intersection of two proper integral ideals containing it (see \cite[Chapter~3]{ns-app}). If we take $I=(-5+S)\cap(-11+S)$ with $S=\langle 5,11,18,24)$, then (1) $\widetilde{I}$ is reflexive by Corollary~\ref{cor:int-irred} and~\ref{D-1}, (2) it is integral by Lemma~\ref{lem:I tilde in S}, and (3) it is irreducible according to \cite[Theorem~4]{ns-app}. However, $I$  is the intersection of four $\mathbb{Z}$-irreducible ideals of $S$ (using \cite[Proposition~24]{ns-app}).

\begin{remark}
If $I$ is reflexive and normalized, from the expression of $I$ given in Proposition \ref{prop:exp-ref}, we deduce that 
\[\operatorname{c}(I)=\max\{-s_1,\dots,-s_n\}+\operatorname{c}(S) = -\min\{s_1,\dots,s_n\}+\operatorname{c}(S).\] 
It follows that 
\[
\operatorname{c}(S)-\operatorname{c}(I)\in S.
\]
Observe that $\operatorname{c}(S)-\operatorname{c}(I)=\operatorname{f}(S)-\operatorname{f}(I)$, and thus this is an alternate proof of \cite[Lemma~4.7]{Spi15}. In particular, every Frobenius number $g$ of a normalized reflexive ideal is a gap of the first type of $S$ (that is, $g\not \in S$ and $\operatorname{f}(S)-g\in S$); whence $g\not\in \operatorname{K}(S)$.
\end{remark}


With this, we are able to prove the following theorem.

\begin{theorem}\label{th:R-irr}
    Let $S$ be a numerical semigroup and let $I$ be a normalized reflexive ideal of $S$. Then, $I$ is NR-irreducible if and only if $\widetilde{I}$ is integrally closed.
\end{theorem}
\begin{proof}
    If $\widetilde{I}$ is integrally closed, then by Lemma~\ref{lem:ic-r-irred} we deduce that $I$ is NR-irreducible.

    For the other implication, we know that as $I$ is reflexive, $I=I^v$, and so, by Lemma~\ref{prop:exp-ref}, we have that $I=(-s_1+S(s_1))\cap \dots \cap (-s_n+S(s_n))$ for some incomparable $s_1,\dots,s_n\in S$. Since $I$ is NR-irreducible, $I=-s_1+S(s_1)$, and so $\widetilde{I}=S(s_1)$ is integrally closed.
\end{proof}

We conclude this section by giving a
different kind of decomposition of 
a reflexive ideal of $S$, in terms of reflexive ideals of over-semigroups of $S$.

  \begin{remark}
  Let $S=T_1 \cap \dots \cap T_r$ be a decomposition of $S$ as an intersection of numerical over-semigroups and $I$ an ideal of $S$. Then, if $I$ is reflexive, there exists $J_1,\dots , J_r$ such that $I=J_1 \cap \dots \cap J_r$, with $J_i$ a $T_i$-reflexive ideal of $T_i$ for all $i\in \{1,\dots,r\}$.
\end{remark}
\begin{proof}
    Since $I$ is reflexive, by Remark~\ref{rem:char-ref} there exists an ideal $J$ such that $I=S-J$, and we have $$I=S-J=(T_1 \cap \dots \cap T_r)-J=(T_1-J) \cap \dots \cap (T_r-J)\,.$$
    For every $j\in\{1,\dots,r\}$, take $J_i=T_i-J$, which is an ideal of $T_i$; indeed, $(T_i-J)+J+T_i \subseteq T_i+T_i=T_i$, thus by \ref{D-6}, we have $(T_i-J)+T_i \subseteq T_i-J$. The claim follows from the fact that $T_i-J$ is clearly $T_i$-reflexive by Remark~\ref{rem:char-ref}. 
\end{proof}
\begin{example}
    The decomposition of $I$ in the previous remark, is not trivial (in general). Take, for example, $S=\langle 7,10,13 \rangle$, $I=\{ 0,29,32\}+S$, $T_1=S \cup \{29\}$, $T_2=S \cup \{32\}$ (these are numerical semigroups). Then $J_1=T_1-(S-I)=I \cup \{16\}$ and $J_2=T_2-(S-I)= I \cup \{25\}$.
\end{example}

\section{Apéry sets}\label{sec:apery-sets}


Let $S$ be a numerical semigroup, and $I$ be an ideal of $S$. Recall that for a given non-zero element $m$ of $S$, the Apéry set of $I$ with respect to $m$ is the set $\Ap(I,m)=\{ i\in I : i-m\not\in I\}$. As in the case of numerical semigroups, this set has the form $\Ap(I,m)=\{a_0,a_1,\dots,a_{m-1}\}$, where $a_i$ is the smallest element in $I$ congruent with $i$ modulo $m$. It easily follows (as in the case of numerical semigroups, see \cite[Proposition~5]{ns-app}) that $\operatorname{f}(I)=\max(\Ap(I,m))-m$. Thus, if $\max(\Ap(S,m))=\max(\Ap(I,m))$, we deduce that $\operatorname{f}(S)=\operatorname{f}(I)$ and $\operatorname{c}(S)=\operatorname{c}(I)$. We will use later this fact.

\begin{remark}
When working with Apéry sets with respect to $m$, we will be referring to the elements in this Apéry sets with indices ranging from $0$ to $m-1$. Thus, if an index is out of this range it will be understood as the corresponding index modulo $m$.   
\end{remark}

It easily follows that if $\Ap(I,m)=\{a_0,a_1,\dots,a_{m-1}\}$ and $z$ is an integer, then $z\in I$ if and only if $a_z\le z$. In particular, if $J$ is an ideal of $S$ and $\Ap(J,m)=\{b_0,b_1,\dots,b_{m-1}\}$, then 
\begin{itemize}
    \item $I\subseteq J$ if and only if $b_i\le a_i$ for all $i\in \{0,\dots,m-1\}$,
    \item $\Ap(I\cup J,m)=\{\min(a_0,b_0),\dots,\min(a_{m-1},b_{m-1})\}$, and 
    \item $\Ap(I\cap J,m)=\{\max(a_0,b_0),\dots,\max(a_{m-1},b_{m-1})\}$.
\end{itemize}

Therefore, if $I=I_1\cap \dots \cap I_k$, with $I_1,\dots,I_k$ ideals of $S$ and this intersection cannot be refined (no redundant factors), then for each $i\in \{1,\dots,k\}$, there exists at least $i_j\in \{0,\dots,m-1\}$ such that $a_{i_j}$ is in $\Ap(I_i,m)$.

\begin{example}
We come back to Example~\ref{ex:8-9-10}. We have 
$S=\langle 8, 9, 10, 11, 12, 15\rangle$, $I=\{ 0, 7, 13, 14 \}+S$, $M-M=S\cup\{7,14\}$ and $S-I= \{ 8, 9, 10, 11, 12, 15 \}+S$.

We already observed that $8\le_{M-M} 15$, and so $15$ is not needed for the representation of $I$ as the intersection of translations of integrally closed ideals. 

If we list the Ap\'ery sets of 
$-s_i+S(s_i)$ for the remaining 
minimal generators of $S-I$,
we get:
\begin{align*}
\Ap(S(8)-8,8)&=\{0,1,2,3,4,13,14,7\}, \\
\Ap(S(9)-9,8)&=\{0,1,2,3,12,13,6,7\},\\
\Ap(S(10)-10,8)&=\{0,1,2,11,12,5,6,7\},\\
\Ap(S(11)-11,8)&=\{0,1,10,11,4,5,6,7\},\\
\Ap(S(12)-12,8)&=\{0,9,10,3,4,5,6,7\}.
\end{align*}


The following table shows that $I  =(-8+S(8))\cap (-10+S(10))\cap (-12+S(12))$ is a 
non-redundant factorization of $I$, because by taking maximums in their respective Apéry sets (ele\-ment-wise), we recover the Apéry set of $I$ and none of them can be removed (we have highlighted the positions where the maximum is achieved).
\[
\begin{array}{lcccccccc}
& 0 & 1 & 2 & 3 & 4 & 5 & 6 & 7 \\ \hline \hline
\Ap(-8+S(8),8)& \mathbf{0} & 1 & 2 & 3 & 4 & \mathbf{13} & \mathbf{14} & \mathbf{7} \\
\Ap(-10+S(10),8) & 0 & 1 & 2 & \mathbf{11} & \mathbf{12} & 5 & 6 & 7  \\
\Ap(-12+S(12),8) & 0 & \mathbf{9} & \mathbf{10} & 3 & 4 & 5 & 6 & 7\\ \hline
\Ap(I,8) & 0 & 9 & 10 & 11 & 12 & 13 & 14 & 7 
\end{array}
\]

Analogously, $I= (-8+S(8))\cap (-9+S(9))\cap (-11+S(11))\cap (-12+S(12))$ is also non-redundant decomposition as the following table shows.
\[
\begin{array}{lcccccccc}
& 0 & 1 & 2 & 3 & 4 & 5 & 6 & 7 \\ \hline \hline
\Ap(-8+S(8),8)& \mathbf{0} & 1 & 2 & 3 & 4 & \mathbf{13} & \mathbf{14} & \mathbf{7} \\
\Ap(-9+S(9),8) & 0 & 1 & 2 & 3 & \mathbf{12} & \mathbf{13} & 6 & 7\\
\Ap(-11+S(11),8) & 0 & 1 & \mathbf{10} & \mathbf{11} & 4 & 5 & 6 & 7  \\
\Ap(-12+S(12),8) & 0 & \mathbf{9} & \mathbf{10} & 3 & 4 & 5 & 6 & 7\\ \hline
\Ap(I,8) & 0 & 9 & 10 & 11 & 12 & 13 & 14 & 7 
\end{array}
\]

\end{example}

The expression of $\Ap(I+J,m)$ is a bit more complicated (see \cite[Proposition~4.8]{icm-apery}). The following result deals with the remaining operation between ideals: $I-J$.

\begin{theorem}\label{th:ap-smi}
Let $S$ be a numerical semigroup and let $m\in S\setminus\{0\}$. Let $I$ and $J$ be ideals of $S$. Let  $\Ap(I,m)=\{a_0,a_1,\dots,a_{m-1}\}$ and $\Ap(J,m)=\{b_0,b_1,\dots,b_{m-1}\}$. Then, 
$\Ap(J-I,m)=\{x_0,x_1,\dots,x_{m-1}\}$, where
\[
x_i=\max\{ b_{i+j}-a_j : j\in \{0,\dots,m-1\}\}.
\]
\end{theorem}
\begin{proof}
    Let $t_i=\max\{ b_{i+j}-a_j : j\in \{0,\dots,m-1\}\}$. We start by showing that $t_i\in S-I$ for all $i\in \{0,\dots,m-1\}$. Suppose that $t_i=b_{i+j}-a_j$ for some $j$. Take $y\in I$. Then, $y=a_k+lm$ for some non-negative integer $l$ and some $k\in \{0,\dots,m-1\}$. Then, $b_{i+j}-a_j+a_k+l m\ge b_{i+k}-a_k+a_k+l m=b_{i+k}+lm\in J$. Both $b_{i+j}-a_j+a_k+l m$ and $b_{i+k}+lm$ are are in the same class modulo $m$, and so this forces  $b_{i+j}-a_j+a_k+l m$ to be in $J$. This proves that $x_i\le t_i$. 

    Suppose that $x_i<t_i$. Then, $x_i+a_j<b_{i+j}$ (the same $j$ as in the preceding paragraph). But this would lead to $x_i+a_j\not\in J$, contradicting that $x_i+I\subseteq J$.
\end{proof}


As a consequence of the previous theorem, we get a description of the Apéry set of $I^v$ in terms of that of $I$.
 
\begin{corollary}\label{cor:ap-smsmi}
    Let $S$ be a numerical semigroup with multiplicity $m$ and let $I$ be an ideal of $S$. Let $\Ap(S,m)=\{0,w_1,\dots,w_{m-1}\}$ and $\Ap(I,m)=\{a_0,a_1,\dots,a_{m-1}\}$. Then, denoting $\Ap(I^v,m)=\{y_0,y_1,\dots,y_{m-1}\}$, we have
\[
y_i=\max\{ w_{i+j}+\min\{a_k-w_{j+k} : k\in \{0,\dots,m-1\}\}  : j\in \{0,\dots,m-1\}\}.
\]
\end{corollary}
\begin{proof}
    Apply Theorem~\ref{th:ap-smi} to $S$ and $S-I$, and take into account that $-\max\{w_{j+k}-a_k : k\in \{0,\dots,m-1\}\}=\min\{a_k-w_{j+k} : k\in \{0,\dots,m-1\}\}$.
\end{proof}

\begin{remark}
    Notice that for every $i$, $w_{i+j}+\min\{ a_k-w_{j+k} : k\in \{0,\ldots m-1\} \}\le w_{i+j}+a_i-w_{j+i}=a_i$. Thus $(y_0,\ldots,y_{m-1})\le (a_0,\dots, a_{m-1})$, yielding $I\subseteq I^v$, which is something already known.
\end{remark}

We can particularize Theorem \ref{th:ap-smi} to $I=J$. 
We get $\Ap(I-I,m)=\{x_0,\dots,x_{m-1}\}$ with $x_i=\max\{ a_{i+j}-a_j : j\in \{0,\dots,m-1\}\}$. This expression allows, for example, to check when the canonical ideal $K$ of $S$ satisfies $K\subseteq I-I$, a condition that we will use later in Proposition~\ref{prop:char-almost-symmetric-ref}.

\begin{corollary}
    Let $S$ be a numerical semigroup and let $m\in S\setminus \{0\}$. Let $I$ be an ideal of $S$ with $\Ap(I,m)=\{a_0,\dots, a_{m-1}\}$, and let $K$ be the canonical ideal of $S$. Let $\operatorname{PF}(S)=\{g_1,\dots,g_t\}$ and set $i_j= (\operatorname{f}(S)-g_j) \bmod m$ for all $j\in\{1,\dots,t\}$. Then, $K\subseteq I-I$ if and only if for every $j\in \{1,\dots, t\}$ and all $k\in \{0,\dots, m-1\}$, the inequality $\operatorname{f}(S)-g_j\ge a_{i_j+k}-a_k$ holds.
\end{corollary}
\begin{proof}
    We know by \cite[Corollary~4.11]{icm-apery} that $K=\operatorname{f}(S)-\operatorname{PF}(S)+S$. Thus, $K\subseteq I-I$ if and only if for every $g\in \operatorname{PF}(S)$, we have that $\operatorname{f}(S)-g\in I-I$. This latter condition is equivalent to $\operatorname{f}(S)-g\ge x_i$, with $i=(\operatorname{f}(S)-g)\bmod m$. The rest of the proof follows by Theorem~\ref{th:ap-smi}.
\end{proof}

Now, another version using the Apéry set of the numerical semigroup.

\begin{corollary}
    Let $S$ be a numerical semigroup with multiplicity $m$. Let $I$ be an ideal of $S$ with $\Ap(S,m)=\{a_0,\dots, a_{m-1}\}$, and let $K$ be the canonical ideal of $S$. Let $\Ap(S,m)=\{0,w_1,\dots,w_{m-1}\}$, and set $f=\operatorname{f}(S)\bmod m$. 
    Then, $K\subseteq I-I$ if and only if for every $i,j\in \{0,\dots, m-1\}$, the inequality $w_f-w_{f-i}\ge a_{i+j}-a_j$ holds.
\end{corollary}
\begin{proof}
    We know by \cite[Proposition~4.11]{icm-apery} that $\Ap(K,m)=\{0,w_f-w_{f-1},\dots,w_f-w_{f-(m-1)}\}$ (where $w_f=f+m$). Thus, $K\subseteq I-I$ if and only if for every $i\in \{0,\dots,m-1\}$, we have   $w_f-w_{f-i}\ge x_i$. The rest of the proof follows by Theorem~\ref{th:ap-smi}.
\end{proof}


We conclude the section by giving  alternate proof of Lemma~\ref{lem:I tilde in S} using Apéry sets.

\begin{corollary}
    Let $S$ be a numerical semigroup and let $I$ be a reflexive ideal of $S$ such that $\operatorname{c}(S)=\operatorname{c}(I)$. Then, $I\subseteq S$.
\end{corollary}
\begin{proof}
    Let $m$ be the multiplicity of $S$. Observe that $\operatorname{c}(I)=\operatorname{c}(S)$ if and only if $\max\operatorname{Ap}(S,m)=\max\operatorname{Ap}(I,m)$. Let $\Ap(S,m)=\{0,w_1,\dots,w_{m-1}\}$ and $\Ap(I,m)=\{0,a_1,\dots,y_{m-1}\}$.  As $I^v=I$, Corollary~\ref{cor:ap-smsmi} ensures that for all $i$, $a_i=\max\{ w_{i+j}+\min\{a_k-w_{j+k} : k\in \{0,\dots,m-1\}\}  : j\in \{0,\dots,m-1\}\}$.

    Suppose that $a_i=\max(\Ap(I,m))$. Then, $w_i=a_i=\max(\Ap(S,m))$, since the maximum of both sets is the same and thus are in the same congruence class modulo $m$. Notice that for $j\neq 0$, $\min_k\{a_k-w_{j+k}\}< 0$. It suffices to take $k$ such that $j+k=i$, and then $a_k-w_i=a_k-a_i<0$. Hence, $w_{i+j}+\min_k\{a_k-w_{j+k}\} < w_{i+j}<w_i$. 

    Now, suppose that there is some $k$ with $a_k-w_k<0$. Then, $w_i+\min_k \{a_k-w_k\}<w_i$. But this yields a contradiction, since $w_i=a_i=\max_j \{w_{i+j}+\min_k \{a_k-w_{j+k}\}\}<w_i$. This forces $a_j\ge w_k$ for all $k$, and consequently $I\subseteq S$.
\end{proof}

\section{Reflexive ideals in some families of numerical semigroups}\label{sec:some-families}

In this section, we give the numerical counterpart of some results in \cite{CDS} that characterize some classes of one-dimensional rings via their reflexive ideals. 

We have mentioned already that if $S$ is a numerical semigroup, then every ideal of $S$ is reflexive. The converse is not difficult to prove and was already shown in \cite[Proposition~I.1.15]{bdf}.

\begin{proposition}[{\cite[Proposition~I.1.15]{bdf}}]
Let $S$ be a numerical semigroup. The following conditions are equivalent:
\begin{enumerate}
    \item $S$ is symmetric;
    \item every ideal of $S$ is symmetric.
\end{enumerate}
\end{proposition}


We now focus on almost symmetric semigroups, a class that generalizes symmetric and pseudo-symmetric semigroups.

Recall that we say that a numerical semigroup $S$ is irreducible if it cannot be expressed as the intersection of two numerical semigroups properly containing it. If $f$ is the Frobenius number of $S$, then irreducibility translates to symmetry (for all $x\in \mathbb{Z}$, $x\not\in S$ implies $f-x\in S$) in the case $f$ is odd or pseudo-symmetry (for all $x\in \mathbb{Z}\setminus\{f/2\}$, $x\not\in S$ implies $f-x\in S$) in the case $f$ is even (see, for instance, \cite[Chapter~2]{ns-app}).

From the definition of $\operatorname{K}(S)$, it follows that $S$ is symmetric if and only if $\operatorname{K}(S)=S$. Hence, by \eqref{eq:prop-canonical}, we have that every ideal of $S$ is reflexive.

In the case, $S$ is pseudo-symmetric, we get $\operatorname{K}(S)=S\cup\{f/2\}$, and every ideal $I$ with $f/2\in I-I$ is reflexive (see \cite[Corollary~14]{ns-app}).

Recall that an integer $x$ is said to be a pseudo-Frobenius number of $S$ if $x+M\subseteq S$, with $M$ the maximal ideal of $S$. The set of pseudo-Frobenius numbers of $S$ is denoted by $\operatorname{PF}(S)$. Clearly, $f\in \operatorname{PF}(S)$. Moreover, $S$ is symmetric if and only if $\operatorname{PF}(S)=\{f\}$, and $S$ is pseudo-symmetric if and only if $\operatorname{PF}(S)=\{f/2,f\}$ (see, for instance, \cite[Chapter~2]{ns-app}).


The cardinality of $\operatorname{PF}(S)$ is known as the \emph{type} of $S$, and is denoted by $\operatorname{t}(S)$. A numerical semigroup is said to be \emph{almost symmetric} if $\operatorname{g}(S)=(\operatorname{f}(S)+\operatorname{t}(S))/2$. In particular, as the symmetry of $S$ is equivalent to $\operatorname{g}(S)=(\operatorname{f}(S)+1)/2$ and the pseudo-symmetry to $\operatorname{g}(S)=(\operatorname{f}(S)+2)/2$ (see \cite[Chapter~2]{ns-app}), we deduce that every irreducible numerical semigroup is almost symmetric.

It is well known that $S$ is almost symmetric if and only if $\operatorname{K}(S)+\operatorname{M}(S)\subseteq \operatorname{M}(S)$ \cite[Proposition~4]{b-f}. Hence, in light of \ref{D-6} $S$ is almost symmetric if and only if $\operatorname{K}(S)\subseteq \operatorname{M}(S)-\operatorname{M}(S)$. 

\begin{proposition}\label{prop:char-almost-symmetric-ref}
    For a numerical semigroup $S$, the following conditions are equivalent:
    \begin{enumerate}
        \item $S$ is almost symmetric;
        \item a non-principal ideal $I$ of $S$ is reflexive if and only if $\operatorname{K}(S) \subseteq I-I$.
    \end{enumerate}
\end{proposition}
\begin{proof}

Let $K=\operatorname{K}(S)$ and let $M=\operatorname{M}(S)$, the maximal ideal of $S$. Recall that $M=S(\operatorname{m}(M))$ is integrally closed. 


Suppose that (1) holds, and let $I$ be a non-principal reflexive ideal of $S$. Since $I$ is not principal, by  Proposition~\ref{prop: J-reflexivity}, $\widetilde{I} \subseteq M$ and so also its integral closure $\bar{I}\subseteq M$ (which also is integrally closed). Therefore, by Proposition~\ref{prop:s-s in e-e} we deduce that $M-M\subseteq I-I$. By hypothesis $S$ is almost symmetric, that is, $K\subseteq M-M$, and thus $K\subseteq I-I$.

If $K\subseteq I-I$, then by~\ref{D-6}, $I+K\subseteq I$ and by Proposition~\ref{prop:bounding-star-canonical}, $I^v\subseteq I+K\subseteq I$, which yields $I^v=I$.

Now, suppose that (2) holds. Since $M$ is integrally closed, by Lemma~\ref{lem:ic-implies-ref}, we deduce that $M$ is reflexive. Hence, by hypothesis $K\subseteq M-M$, which means that $S$ is almost symmetric.
\end{proof}

\begin{theorem}\label{thm:AG}
    For a numerical semigroup $S$ with maximal ideal $M$, the following conditions are equivalent:
    \begin{enumerate}
        \item $S$ is almost symmetric;
        \item a non-principal ideal $I$ of $S$ is reflexive if and only if $I$ is an ideal of $M-M$;
        \item a non-principal integral ideal $I$ of $S$ is reflexive if and only if $I$ is an integral ideal of $M-M$.
    \end{enumerate}
\end{theorem}
\begin{proof}
    First, notice that $M-M$ is an over-semigroup of $S$ (Remark~\ref{rem:I-I-oversemigroup}). Let $K$ be the canonical ideal of $S$.
    
    \emph{(1) implies (2)}. Let $I$ be a non-principal ideal of $S$ such that $I^v=I$. We have to prove that $I+(M-M)=I$. As in the proof of Proposition~\ref{prop:char-almost-symmetric-ref}, we deduce that $M-M\subseteq I-I$, and so $I\subseteq I+(M-M)\subseteq I+(I-I)\subseteq I$.

    The other implication follows from Theorem~\ref{thm:A-Valentina}, and by the fact that $K\subseteq M-M$ because $S$ is almost symmetric.

    \emph{(2) implies (1)}. We use Theorem~\ref{thm:A-Valentina} again. Let $I$ be an ideal of $M-M$. 
    As $M-M$ is an over-semigroup of $S$, we deduce that $I$ is an ideal of $S$.
    If $I$ is principal (as an ideal of $S$), then it is reflexive, and if it is not principal, then by hypothesis, $I$ is reflexive. Thus, every ideal of $M-M$ is reflexive as an ideal of $S$. By Theorem~\ref{thm:A-Valentina}, we conclude that $K\subseteq M-M$.

    \emph{(2) if and only if (3)}. The proof follows by translations.
\end{proof}

The next results concern the class of numerical semigroups with maximal embedding dimension and its subclass of Arf semigroups.


It is well-known that for any numerical semigroup $S$, one has $\operatorname{e}(S)\le \operatorname{m}(S)$, and if equality holds, $S$ is said to have \emph{maximal embedding dimension}.\\
An ideal $I$ of $S$ is said to be \emph{stable} if $I+I=I+\min(I)$. Stability is invariant under translations.
If every integrally closed ideal of $S$ is stable, $S$ is said to be an \emph{Arf numerical semigroup}.

Notice that the stability of $I$ is equivalent to $I-I=I-\min(I)$ and that 
$S$ has maximal embedding dimension if and only if $M$ is stable (cf. e.g. \cite{ns-app} page 46). Hence, every Arf numerical semigroup has maximal embedding dimension.

%


\begin{theorem}\label{thm:MED}
    For a numerical semigroup $S$ with maximal ideal $M$, the following conditions are equivalent:
    \begin{enumerate}
        \item $S$ has maximal embedding dimension;
        \item a non principal ideal $I$ of $S$ is reflexive if and only if $I$ is a reflexive ideal of $M-M$;
        \item $M$ is $(M-M)$-reflexive.
    \end{enumerate}
\end{theorem}
\begin{proof}
Let $m$ be the multiplicity and $f$ be the Frobenius number of $S$.

\emph{(1) implies (2)} The numerical semigroup $S$ has maximal embedding dimension if and only if $M$ is stable, that is, $M-M=M-m$. This fact immediately implies that any ideal of $S$ is also an ideal of $M-m=M-M$. Moreover, from Proposition
~\ref{prop: J-reflexivity}, we know that a reflexive non principal ideal $I$ is also
$M$-reflexive, and therefore $(M-m)-((M-m)-I)=M-(M-I)=I$ (where the first equality follows by \ref{D-1}), i.e. $I$ is a reflexive ideal of $M-M$.


\emph{(2) implies (3)} This implication is clear since $M$ is always reflexive (Lemma~\ref{lem:ic-implies-ref}) and non-principal.
    
\emph{(3) implies (1)} Notice that $M+(M-2M) \subseteq \mathbb{N}$, indeed $\min(M-2M) \geq -m$, otherwise $2m+\min(M-2M) < m$ which is a contradiction since $2m+\min(M-2M) \in M$.

We shall now prove that $M$ is stable if and only if $0 \in M+(M-2M)$. If $2M=M+m$, 
then $$M+(M-2M)=M+(M-(M+m))=(M+(M-M))-m=M-m \ni 0\,.$$ Conversely, if $0 \in M+(M-2M)$, there exists $x \in M$ such that $-x \in M-2M$, thus $2M \subseteq M+x$ and so $2M=M+x$. This implies $2m=\min(2M)=\min(M+x)=m+x$, i.e. $2M=M+m$.

Therefore, if $S$ does not have maximal embedding dimension, $0 \not\in M+(M-2M)$ and thus $f+(M+(M-2M)) \subseteq M$ which yields $$f \in M-(M+(M-2M)) =(M-M)-((M-M)-M)$$ (here we repeatedly use \ref{D-5}) and thus $M$ is not $(M-M)$-reflexive.
\end{proof}





\begin{example}
The following examples (\cite[Example 5.5]{CDS}) illustrate the difference between Theorem~\ref{thm:AG} and Theorem~\ref{thm:MED}.
    \begin{enumerate}
        \item Let $S_1=\langle 4,5,7 \rangle$ and $M=\operatorname{M}(S_1)$. We have $\operatorname{K}(S_1)=\{0,3\}+S_1$ and a direct computation shows that $\operatorname{K}(S_1)\subseteq M-M$, thus $S_1$ is almost symmetric but it does not have maximal embedding dimension. In fact, the ideal $I=\{5,7,8\}+S_1$ is $S_1$-reflexive but not $(M-M)$-reflexive as $6 \in ((M-M)-((M-M)-I))\setminus I$ (the existence of such ideal is ensured by Theorem~\ref{thm:MED}).
        \item Consider the maximal embedding dimension numerical semigroup $S_2=\langle 3,7,8 \rangle$ and let $I=\{6,8,10\}+S_2$. A direct computation shows that $I$ is an ideal of $\operatorname{M}(S_2)-\operatorname{M}(S_2)$, but $I$ is not $S_2$-reflexive since $7 \in I^v\setminus I$. Therefore, by Theorem~\ref{thm:AG}, $S_2$ is not almost symmetric.
    \end{enumerate}
\end{example}



\begin{corollary}\label{cor:Arf}
    For a numerical semigroup $S$, the following conditions are equivalent:
    \begin{enumerate}
        \item $S$ is Arf;
        \item an ideal $I$ of $S$ is reflexive if and only if $I$ is the translation of an integrally closed ideal;
        \item an ideal $I \in \mathfrak{I}_0(S)$ is reflexive if and only if $I$ is NR-irredudible.
    \end{enumerate}
\end{corollary}
\begin{proof}
    \emph{(1)  implies (2)}. Every integrally closed ideal is reflexive by Lemma~\ref{lem:ic-implies-ref}, therefore we need to prove the converse. Let $I \subseteq S$ be a reflexive ideal of $S$. Since $S$ is Arf, every integrally closed ideal is stable, and thus $\bar{I}=\min(\bar{I})+(\bar{I}-\bar{I})$. Therefore, applying Corollary~\ref{cor:m+bar I-bar I}, we have
    $\bar{I}=\min(\bar{I})+(\bar{I}-\bar{I}) \subseteq \widetilde{I} \subseteq \bar{I}$. It follows that $\widetilde{I}=\bar{I}$. Thus, $\widetilde{I}$ is integrally closed
    and so is $I$.
    
    \noindent \emph{(2) implies (1)}. Let $I$ be an integrally closed ideal of $S$. Then, $I=\bar{I}=\widetilde{I}$ and by Lemma~\ref{lem:s-i=bar i-i}, $\min(I)+(I-I)=\min(I)+(S-I)$. Hence, by Remark~\ref{rem:char-ref}, $\min(I)+(I-I)$ is reflexive and thus it is integrally closed by assumption. Since $\min(I-I)=0$, it follows that $I=\min(I)+(I-I)$, which is in turn equivalent to the stability of $I$. Therefore every integrally closed ideal of $S$ is stable, whence $S$ is Arf.

    The equivalence between (2) and (3) follows from Theorem~\ref{th:R-irr}.
\end{proof}




\section{Trace ideals}\label{sec:trace-ideals}

In recent years, the study of trace ideals and their relation with the ring and other families of modules (e.g. canonical modules, reflexive modules, stable ideals) has attracted considerable attention (see for instance \cite{HHS, HL22, GIK20, EG23}). Indeed, these ideals, as well as reflexive ones, encode useful information about the ring. The results in this section are inspired by the ones in \cite{GIK20, EG23, IK24}, in which the connections between rings, their trace ideals and their reflexive modules are studied in a commutative and not necessarily Noetherian setting.

In this section, we study a correspondence between the set of reflexive trace ideals of a numerical semigroup and its reflexive over-semigroups, and moreover, prove how trace ideals can be used to characterize Arf numerical semigroups in two different ways, both related to reflexivity.

Given an ideal $I$ of $S$, the \emph{trace (ideal)} of $I$ is the ideal $\operatorname{tr}(I)=I+(S-I)$. An ideal $J$ is said to be a \emph{trace ideal} if there exists an ideal $I$ of $S$ such that $J=\operatorname{tr}(I)$.
Notice that $\operatorname{tr}(I)$ is invariant under translations of $I$ (by \ref{D-1}), and if $I$ is a trace ideal, then $I \subseteq S$.


\begin{remark}\label{rem:i=i+i-i}
    For any $I \in \mathfrak{I}(S)$, we have $I=I+(I-I)$. Indeed, $I=I+0 \subseteq I+(I-I) \subseteq I$.
\end{remark}

By Remark~\ref{rem:I-I-oversemigroup}, this last assertion has a different rephrasing: $I$ is always an ideal of $I-I$. 

\begin{proposition}\label{prop:i=tri}
     If $J$ is a trace ideal of $S$, then $J=\operatorname{tr}(J)$.
\end{proposition}
\begin{proof}
    Let $I \in \mathfrak{I}(S)$ such that $J=\operatorname{tr}(I)$, then
    \begin{align*}
        \operatorname{tr}(J)&=\operatorname{tr}(I)+(S-\operatorname{tr}(I))\\[0.25em]
        &=I+(S-I)+(S-(I+(S-I)))\\[0.25em]
        &=I+(S-I)+((S-I)-(S-I))\\[0.25em]
        &=I+(S-I)\\[0.25em]
        &=\operatorname{tr}(I)\\[0.25em]
        &=J,
    \end{align*}
    where the third equality follows from~\ref{D-5} and the fourth by Remark~\ref{rem:i=i+i-i} applied to $S-I$.
\end{proof}

Thus, we have that $I$ is a trace ideal if and only if 
\[\operatorname{tr}(I)=I.\]


The numerical semigroup analogue to \cite[Proposition~A.1]{HHS} is the following.

\begin{proposition}\label{prop:cond tr(I)}
    Let 
    $I \in \mathfrak{I}(S)$. Then 
    \[\operatorname{C}(S) \subseteq \operatorname{tr}(I) \subseteq S.\] 
\end{proposition}
\begin{proof}
    Let $I \in \mathfrak{I}(S)$. The inclusion $\operatorname{tr}(I) \subseteq S$ follows from the definition of trace. In order to prove the other inclusion, we shall first compute the conductor ideal of $S-I$; applying equation \eqref{eq:conductor-quotient} of Section 2,~\ref{D-5} twice, and \ref{D-1} we get
    \begin{align*}
        \operatorname{C}(S-I)&=(S-I)-\mathbb{N}\\
        &=S-(I+\mathbb{N})\\
        &=(S-\mathbb{N})-I\\
        &=(\operatorname{c}(S)+\mathbb{N})-I\\
        &=\operatorname{c}(S)+(\mathbb{N}-I)\\
        &=\operatorname{c}(S)-\min(I)+\mathbb{N},
    \end{align*}
    where the last equality follows from the fact that $\mathbb{N}-I=-\min(I)+\mathbb{N}$. Indeed, clearly, $\min(I)+(\mathbb{N}-I) \subseteq \mathbb{N}$ and so $\mathbb{N}-I\subseteq -\min(I)+\mathbb{N}$. For the other inclusion, if $x \in \mathbb{Z}$ is such that $x \geq -\min(I)$, then $x+i \geq 0$ for all $i \in I$ and thus $x+I \subseteq \mathbb{N}$, whence $x\in \mathbb{N}-I$. This proves the equality. Therefore, the following holds 
    \begin{align*}
        \operatorname{C}(S)&=\min(I)+(\operatorname{c}(S)-\min(I)+\mathbb{N})\\
        &\subseteq I+(\operatorname{C}(S-I))\\
        &\subseteq I+(S-I)\\
        &=\operatorname{tr}(I).
    \end{align*}
    We conclude that $\operatorname{C}(S) \subseteq\operatorname{tr}(I) \subseteq S$ for any $I \in \mathfrak{I}(S)$.
\end{proof}

In light of the previous Proposition and of the fact that if $C \subseteq I \subseteq S$, then $ C \subseteq I^v \subseteq S$ (by Lemma~\ref{lem:I tilde in S}), it is natural to compare the properties of being reflexive and a trace ideal. In general, the properties of being trace and being reflexive do not imply each other, as the next example shows.

\begin{example}\label{ex:tr-v}
Let $S=\langle 7,10,13\rangle$ and let $I_1=\{7,36,39\}+S$, $I_2=\{7,26\}+S$, $I_3=\{7,23,26\}+S$, and $I_4=\{7,13\}+S$.
The following computations using \texttt{NumericalSgps} \cite{numericalsgps} show that $I_1$ is reflexive but not a trace ideal, $I_2$ is neither reflexive nor a trace ideal, $I_3$ is both reflexive and a trace ideal, and $I_4$ is not reflexive but it is a trace ideal.
\begin{verbatim}
gap> s:=NumericalSemigroup(7,10,13);;
gap> i1:=[7, 36, 39] + s;;
gap> s-(s-i1)=i1;   #is reflexive?
gap> i1=i1+(s-i1);  #is trace?
true
false
gap> i2:=[7, 26] + s;;
gap> s-(s-i2)=i2; i2=i2+(s-i2);
false
false
gap> i3:=[7, 23, 26] + s;;
gap> s-(s-i3)=i3; i3=i3+(s-i3);
true
true
gap> i4:=[7, 13]+s;;
gap> s-(s-i4)=i4; i4=i4+(s-i4);
false
true
\end{verbatim}
\end{example}


\begin{remark}\label{rem:int closed -> trace}
    Every integrally closed ideal $J \in \mathfrak{I}(S)$ which contains the conductor ideal is a trace ideal since, by Lemma~\ref{lem:s-i=bar i-i}, $S-J=J-J$, and then by Remark~\ref{rem:i=i+i-i}, we have $J+(S-J)=J+(J-J)=J$. 
\end{remark}

Example~\ref{ex:tr-v} also shows that not all integral ideals containing the conductor ideal are trace ideals, while Proposition~\ref{prop:cond tr(I)}, shows that we can focus on integral ideals containing the conductor ideal. The following lemma illustrates a characteristic property of trace ideals, and, moreover, it generalizes Remark~\ref{rem:int closed -> trace}, showing that to be trace it is a weaker condition than to be integrally closed.

\begin{lemma}
    An ideal $I$ such that $\operatorname{C}(S) \subseteq I \subseteq S$ is a trace ideal if and only if $I-I=S-I$. 
\end{lemma}
\begin{proof}
    Notice that, from the choice of $I$, we have that $I-I \subseteq S-I$ by~\ref{D-2}. If $I$ is a trace ideal, then $I=\operatorname{tr}(I)$ by Proposition~\ref{prop:i=tri} and the inclusion $I+(S-I) \subseteq I$ holds, therefore $S-I \subseteq I-I$ by \ref{D-6}. 
    
    For the other implication, if $S-I=I-I$, then $\operatorname{tr}(I)=I+(S-I)=I+(I-I)=I$ by Remark~\ref{rem:i=i+i-i}, and therefore $I$ is a trace ideal.
\end{proof}

\begin{lemma}\label{lem: I-I}
    If $I \in \mathfrak{I}(S)$ is reflexive, then $I-I=S-\operatorname{tr}(I)$. In particular, if $I$ is reflexive, then $I-I$ is reflexive.
\end{lemma}
\begin{proof}
    The reflexivity of $I$ implies $I-I=(S-(S-I))-I$, therefore $I-I=S-(I+(S-I))=S-\operatorname{tr}(I)$ by~\ref{D-5}.
    The last assertion follows from Remark~\ref{rem:char-ref}. 
\end{proof}

Let $S$ be a numerical semigroup. Let $C=\operatorname{C}(S)$ and let $\mathcal{O}(S)$ be the set of numerical over-semigroups of $S$. Consider the sets
\[
A=\{I\in\mathfrak{I}(S)\,:\,I=I^v \text{ and } C \subseteq I \subseteq S\}\quad\text{and}\quad B=\{T \in\mathcal{O}(S)\,:\,T^v=T\},
\]
and define
\[
\begin{aligned}
\Phi: &\ A \longrightarrow B, &\quad& \Psi: B \longrightarrow A, \\
   &\,\,I \longmapsto I-I,       &    &\quad\,\,T \longmapsto S-T.
\end{aligned}
\]
Both maps are well-defined. In fact, if $I$ is reflexive, then $I-I$ is a reflexive numerical over-semigroup of $S$ by Remark~\ref{rem:I-I-oversemigroup} and Lemma~\ref{lem: I-I}, and if $T \in \mathcal{O}(S)$ is reflexive, we have $S \subseteq T \subseteq\mathbb{N}$, thus by \eqref{eq:conductor-quotient} and~\ref{D-2}, $C=S-\mathbb{N} \subseteq S-T \subseteq S-S=S$; $S-T$ is a reflexive ideal of $S$ by Remarks~\ref{rem:char-ref} and~\ref{rem:oversemigroup-is-ideal}.  

The following result is a numerical version of \cite[Lemma 2.6]{GIK20}.

\begin{proposition}\label{prop:phi and psi}
    We have $\Phi\circ \Psi=\operatorname{id}_B$, in particular $\Psi$ is injective and $\Phi$ is surjective. Moreover, the image of $\Psi$ is the set of reflexive trace ideals of $S$. 
\end{proposition}
\begin{proof}
    Let $T \in B$. Then 
    \begin{multline*}
     \Phi(\Psi(T))=\Phi(S-T)=(S-T)-(S-T)=S-(T+(S-T))\\
            =(S-(S-T))-T=T-T=T,
    \end{multline*}
    where the third and fourth equalities follow from~\ref{D-5} and the last  from the fact that $T$ is a numerical semigroup. This proves $\Phi \circ \Psi=\operatorname{id}_B$.
    
    As for the second assertion, we need to show that 
    \[\operatorname{im}\Psi=\{I \in \mathfrak{I}(S)\,:\,I=I^v=\operatorname{tr}(I)\}\]
    (the set on the right-hand side is non-empty by Remark~\ref{rem:int closed -> trace} and by Proposition~\ref{prop:cond tr(I)} this set is contained in $A$). If $I$ is a reflexive trace ideal, then $I-I\in B$ and $\Psi(I-I)=S-(I-I)=S-(S-\operatorname{tr}(I))$ (by Lemma~\ref{lem: I-I}). Therefore, $\Psi(I-I)=S-(S-\operatorname{tr}(I))=S-(S-I)=I$ by the reflexivity of $I$. Hence, $\{I \in \mathfrak{I}(S)\,:\,I=I^v=\operatorname{tr}(I)\} \subseteq \operatorname{im}\Psi$. Conversely, for every reflexive $T \in \mathcal{O}(S)$, $\Psi(T)=S-T$ is a trace ideal since by Remark~\ref{rem:S-T-ideal-T}, $S-T$ is an ideal of $T$, and so $S-T=T+(S-T)=\operatorname{tr}(T)$.
    Thus the claim follows.
\end{proof}


\begin{remark}\label{rem:C x+(I-I)}
Let $I \in \mathfrak{I}(S)$ and $x=\min(I)$, then, applying \eqref{eq:conductor-quotient},~\ref{D-5}, and~\ref{D-1},
\begin{align*}
\operatorname{C}(x+(I-I))&=(x+(I-I))-\mathbb{N}\\
&=x+(I-(I+\mathbb{N}))\\
&=x+(I-(x+\mathbb{N}))\\
&=I-\mathbb{N}\\
&=\operatorname{C}(I)\,.
\end{align*}    
\end{remark}

The following remark is a numerical version of \cite[Corollary 2.2]{GIK20}.
\begin{remark}\label{rem:i=j}
    Let $I, J \in \mathfrak{I}(S)$. If $I=t+J$ for some $t \in \mathbb{Z}$ and $\operatorname{c}(I)=\operatorname{c}(J)$, then $I=J$, that is, $t=0$. Indeed, assume $t \neq 0$. If $t<0$, then $\operatorname{c}(J)+t \leq \operatorname{c}(J)-1=\operatorname{c}(I)-1=\operatorname{f}(I)$ and $\operatorname{c}(J)+t+\mathbb{N} \subseteq t+J=I$, but this implies $\operatorname{f}(I) \in I$, which is a contradiction. The case $t>0$ follows by applying the same argument to $J=-t+I$.
\end{remark}

\begin{proposition}\label{prop:phi is inj}
For a numerical semigroup $S$, the following are equivalent:
\begin{enumerate}
    \item $S$ is Arf;
    \item the map $\Phi$ is injective.
\end{enumerate}
\end{proposition}
\begin{proof}
    \emph{(1)  implies (2)}. Assume that $S$ is Arf. 
    By Corollary~\ref{cor:Arf}, if $I\in A$, then $I$ is the translation of an integrally closed ideal. Let $C=\operatorname{C}(S)$. The condition $C\subseteq I\subseteq S$ forces $\operatorname{c}(S)=\operatorname{c}(I)$, and so $I$ is an integrally closed ideal (by Remark~\ref{rem:i=j}). The converse is also true, since we know that every integrally closed ideal $I$ is reflexive (Lemma~\ref{lem:ic-implies-ref}) and thus if it contains $C$, we deduce $I\in A$. Therefore, $A$ is the set of integrally closed ideals of $S$ containing $C$.
    
    Fixed $I,J \in A$, we have that $I=\min(I)+(I-I)$ and $J=\min(J)+(J-J)$, because every integrally closed ideal of $S$ is stable. Thus, if $I$ and $J$ are such that $I-I=\Phi(I)=\Phi(J)=J-J$, we have 
    \[
    I=\min(I)+(I-I)=\min(I)-\min(J)+\min(J)+(J-J)=\min(I)-\min(J)+J.
    \]
    In particular, there exists $t \in \mathbb{Z}$ such that $I=t+J$, but, since $\operatorname{c}(I)=\operatorname{c}(J)=\operatorname{c}(S)$, we must have $t=0$ by Remark~\ref{rem:i=j}. Hence, $I=J$, that is, $\Phi$ is injective.

\noindent \emph{(2)  implies (1)}. Fix an integrally closed ideal $J$ of $S$, we prove that $J$ is stable in order to show that $S$ is Arf. If $J=S(s)$ with $s\ge \operatorname{c}(S)$, then $J-J=\mathbb{N}=J-s$, indeed it is clear that $J-s=\mathbb{N}$; moreover $J-J \subseteq \mathbb{N}$ (by Remark~\ref{rem:I-I-oversemigroup}), and for every $n \in \mathbb{N}$ and $x \in J$, $x+n \ge s$, thus $J+\mathbb{N} \subseteq J$, i.e. $\mathbb{N} \subseteq J-J$. In particular $J$ is stable.
If $s<\operatorname{c}(S)$, then we prove that $J \in A$. Indeed, by definition, $C \subseteq J \subseteq S$ and $J$ is reflexive by Lemma~\ref{lem:ic-implies-ref}. The ideal $J-J$ is reflexive by Lemma~\ref{lem: I-I}, thus the ideal $\min(J)+(J-J)$ is reflexive too, being a translation of $J-J$. By definition of $J-J$, we have $\min(J)+(J-J) \subseteq J \subseteq S$, and we also have that $\operatorname{c}(\min(J)+(J-J))=\operatorname{c}(J)=\operatorname{c}(S)$ by Remark~\ref{rem:C x+(I-I)}. In particular, $\min(J)+(J-J) \in A$. Moreover,
     \begin{align*}
        \Phi(\min(J)+(J-J))&=(\min(J)+(J-J))-(\min(J)+(J-J))\\[0.25em]
        &=(J-J)-(J-J)\\[0.25em]
        &=J-(J+(J-J))\\[0.25em]
        &=J-J\\[0.25em]
        &=\Phi(J)
    \end{align*}
    (the third equality follows from~\ref{D-5}, and the fourth by Remark~\ref{rem:i=i+i-i}). Thus, since $\Phi$ is assumed to be injective, we have $\min(J)+(J-J)= J$. Therefore, every integrally closed ideal of $S$ is stable and hence $S$ is Arf.
\end{proof}

In \cite[Proposition 3.2]{IK24}, the authors proved that, if $I$ is a regular trace ideal of an Arf semi-local ring $R$, then $I$ is integrally closed. Moreover, in \cite[Theorem 3.5]{IK24}, it is proven that if $R$ is Arf and $I$ is a regular ideal of $R$, then being isomorphic to an integrally closed ideal, being a reflexive $R$-module and being isomorphic to a trace ideal, are equivalent conditions for $I$. The following proposition (together with Corollary~\ref{cor:Arf}) establishes that this equivalence characterizes Arf numerical semigroups rings. 

\begin{proposition}\label{prop:arf, ref, tr}
    For a numerical semigroup $S$, the following conditions are equivalent:
    \begin{enumerate}
        \item $S$ is Arf;
        \item an ideal $I$ of $S$ is reflexive if and only if $I$ is the translation of a trace ideal.
    \end{enumerate}
\end{proposition}
\begin{proof}
\emph{(1) implies (2)}. 
Assume that $S$ is Arf. Let $I$ be a reflexive ideal of $S$. From Corollary~\ref{cor:Arf}, it follows that $\widetilde{I}$ is an integrally closed ideal of $S$, and then, by Remark~\ref{rem:int closed -> trace}, we have that $\widetilde{I}$ is a trace ideal. It follows that any reflexive ideal of $S$ is the translation of a trace ideal.

Let $I$ be an ideal of $S$ such that $I$ is a translation of a trace ideal. Then, there exist $t \in \mathbb{Z}$ and a trace ideal $H$ of $S$ such that $t+\widetilde{I}=H$. By Proposition~\ref{prop:i=tri} and Proposition~\ref{prop:cond tr(I)}, we have that $\operatorname{c}(H)=\operatorname{c}(\operatorname{tr}(H))=\operatorname{c}(S)=\operatorname{c}(\widetilde{I})$, hence $\widetilde{I}=H$ by Remark~\ref{rem:i=j}. This implies that $\widetilde{I}=\operatorname{tr}(\widetilde{I})$. Let $x=\min(\widetilde{I})=\min(\bar{I})$.  
From~\ref{D-2} and the fact that every integrally closed ideal is stable $S$ is Arf), we have
\[
\widetilde{I}=\widetilde{I}+(S-\widetilde{I})\supseteq \widetilde{I}+(\bar{I}-\widetilde{I}) \supseteq x+(\bar{I}-\bar{I})=\bar{I}
\]
which implies that $\widetilde{I}=\bar{I}$, because the inclusion $\widetilde{I} \subseteq \bar{I}$ always holds. Thus $\widetilde{I}=\widetilde{I}^v$ and, in particular, we have that $I$ is reflexive.

\noindent \emph{(2) implies (1)}. Notice that for any integral ideal $I$ of $S$, we have $I+(I-I)=I \subseteq S$ from Remark~\ref{rem:i=i+i-i}, therefore $I \subseteq S-(I-I)$ by~\ref{D-6}. Let $J$ be an integrally closed ideal of $S$ and $x=\min(J)$. As in the proof of Proposition~\ref{prop:arf, ref, tr}, we may assume that $x<\operatorname{c}(S)$, and consequently $\operatorname{c}(J)=\operatorname{c}(S)$. The ideal $x+(J-J)$ is integral by definition of $J-J$ and reflexive (by Lemma~\ref{lem: I-I}). Also, by Remark~\ref{rem:C x+(I-I)}, $\operatorname{c}(J)=\operatorname{c}(x+(J-J))$, and as $J$ is integrally closed, this means that $\operatorname{c}(x+(J-J))=\operatorname{c}(S)$.
Then, by the assumption, it follows that $x+(J-J)$ is the translation of a trace ideal $H$, say $H=t+x+(J-J)$ for some $t \in \mathbb{Z}$. By Proposition~\ref{prop:cond tr(I)} and Remark~\ref{rem:C x+(I-I)}, $\operatorname{c}(\operatorname{tr}(H))=\operatorname{c}(H)=\operatorname{c}(S)=\operatorname{c}(x+(J-J))$ and therefore $t=0$ by Remark~\ref{rem:i=j}. This implies that $x+(J-J)$ is itself a trace ideal, hence by~\ref{D-1},
\[
x+(J-J)=(x+(J-J))+(S-(x+(J-J)))=(J-J)+(S-(J-J)).
\]
Observe that $J+(J-J)\subseteq J\subseteq S$, and by \ref{D-6}, $J\subseteq S-(J-J)$.
Thus, we have $J=0+J \subseteq (J-J)+(S-(J-J))=x+(J-J)$, which means that $J$ is stable. Since $J$ was arbitrarily chosen among integrally closed ideals, it follows that $S$ is Arf.
\end{proof} 

Notice that in the proof of the second implication we only used the fact that reflexive ideals are trace (and not the converse), thus the previous proposition proves that a version of \cite[Corollary 7.5]{HL22} holds in the case of numerical semigroup rings.

\section*{Acknowledgements}

The first author is supported by the Polish OPUS grant 28 NCN-2024/55/B/ST1/01437 - The structure of perfect ideals and linkage. The second author is supported by the grant PRIN 2022 DD n. 104/2022 e DD n. 1401/2024 del MUR (code 2022K48YYP). The third author acknowledges financial support from grant PID2022-138906NB-C21, funded by MICIU/AEI/10.13039/501100011033 and by ERDF ``A way of making Europe''.

\section*{Declaration}
On behalf of all authors, the corresponding author states that there is no conflict of interest.

\end{document}